\documentclass[journal,twoside,web]{ieeecolor}
\usepackage{generic}
\usepackage{cite}
\usepackage{amsmath,amssymb,amsfonts}
\usepackage{algorithmic}
\usepackage{graphicx}
\usepackage{textcomp}
\usepackage{stmaryrd}
\usepackage{graphics}
\usepackage{algorithm}
\usepackage{mathrsfs} 
\usepackage{epsfig}
\usepackage{indentfirst}
\usepackage{color}
\usepackage{xcolor}
\usepackage{epstopdf}
\usepackage{float}
\usepackage{amstext}
\usepackage{amsfonts}
\usepackage{booktabs}
\usepackage{flushend}
\usepackage{enumerate}  
\usepackage{multirow}
\usepackage{makecell}
\usepackage{cases}
\newtheorem{definition}{Definition}
\newtheorem{theorem}{Theorem}
\newtheorem{lemma}{Lemma}
\newtheorem{proposition}{Proposition}
\newtheorem{corollary}{Corollary}
\newtheorem{remark}{Remark}
\newtheorem{assumption}{Assumption}

\usepackage{enumerate}
\usepackage[colorlinks, linkcolor=black, anchorcolor=black, citecolor=black]{hyperref}
\def\BibTeX{{\rm B\kern-.05em{\sc i\kern-.025em b}\kern-.08em
		T\kern-.1667em\lower.7ex\hbox{E}\kern-.125emX}}
\begin{document}
\title{A cutting-surface consensus approach for distributed robust optimization of multi-agent systems}
\author{Jun Fu, \IEEEmembership{Senior Member, IEEE}, and Xunhao Wu
\thanks{This paper was supported in part by the National Nature Science Foundation of China under Grant 61825301. {The authors contributed equally to this work. (Corresponding author: Jun Fu.)}}
\thanks{The authors are with the State Key Laboratory of Synthetical Automation for Process Industries, Northeastern University, Shenyang 110819, China (e-mail: junfu@mail.neu.edu.cn, neuwxh2102043@163.com).}}

\maketitle

\begin{abstract}
  A novel and fully distributed optimization method is proposed for the distributed robust convex program (DRCP) over a time-varying unbalanced directed network under the uniformly jointly strongly connected (UJSC) assumption. Firstly, an approximated DRCP (ADRCP) is introduced by discretizing the semi-infinite constraints into a finite number of inequality constraints to ensure tractability and restricting the right-hand side of the constraints with a positive parameter to ensure a feasible solution for (DRCP) can be obtained. This problem is iteratively solved by a distributed projected gradient algorithm proposed in this paper, which is based on epigraphic reformulation and gradient projected operations. Secondly, a cutting-surface consensus approach is proposed for locating an approximately optimal consensus solution of the DRCP with guaranteed local feasibility for each agent. This approach is based on iteratively approximating the DRCP by successively reducing the restriction parameter of the right-hand constraints and adding the cutting-surfaces into the existing finite set of constraints. Thirdly, to ensure finite-time termination of the distributed optimization, a distributed termination algorithm is developed based on consensus and zeroth-order stopping conditions under UJSC graphs. Fourthly, it is proved that the cutting-surface consensus approach terminates finitely and yields a feasible and approximate optimal solution for each agent. Finally, the effectiveness of the approach is illustrated through a numerical example.
\end{abstract}


\section{Introduction}
\label{sec:introduction}
\IEEEPARstart{D}{istributed} optimization refers to an optimization strategy for multi-agent systems in which agents agree on a consensus and feasible decision that guarantees the optimality of a global objective through local computation and agent communication. Distributed optimization has been extensively studied \cite{yang2019survey,nedic2018distributed,molzahn2017survey} and has been applied in a wide range of fields such as multi-robot system scheduling \cite{FANG2019103}, unmanned aerial vehicle (UAV) formation flight \cite{9869349}, cooperative localization in sensor networks \cite{8755514}, network resource allocation \cite{iiduka2018distributed}, and economic dispatch of smart grids \cite{9484074}. Uncertain distributed optimization is considerably important, e.g., to guarantee stability or to ensure safety \cite{ben2009robust}. Uncertainties considered in distributed optimization include inaccurate gradient information \cite{fallah2019robust}, communication noise \cite{iakovidou2022s}, communication time delay \cite{zhao2018analysis}, data packet loss \cite{wu2017distributed}, constraints affected by uncertainty, etc. This article focuses on distributed optimization problems with uncertain constraints.
\par
Techniques for dealing with distributed optimization with uncertain constraints have been widely researched. In the early work, Lee and Nedic \cite{6461383,lee2015asynchronous} introduced distributed random projection algorithms by designing random projections to ensure almost sure feasibility of the solution and using a distributed subgradient algorithm \cite{nedic1} to cooperatively minimize the global objective. In \cite{you2018distributed,8022966,falsone2020scenario,carlone2014distributed,chamanbaz2017randomized,Pantazis2022}, some scenario-based algorithms were proposed, which are based on extracting a large number of scenarios from the uncertainty set to obtain a scenario problem (SP), and then applying existing constrained distributed optimization methods to solve the SP. These algorithms provide probabilistic feasibility of the resulting solutions. The above two algorithms require the realizations of the uncertainty parameters to be randomly sampled from an underlying distribution. Thus, they unavoidably fail to find a feasible solution for all values of the uncertainty but can provide feasibility assurance in a probabilistic sense. However, in some practical applications, uncertain constraints are safety-critical, and hence, all agents need to locate guaranteed feasible and optimal consensus solutions to ensure that the multi-agent system has a good robustness performance. 
\par 
To address such issues, this paper considers a distributed robust convex program (DRCP) where the inequality constraints are influenced by bounded uncertainty. According to the treatment of uncertainty, the existing methods are categorized into two main types: the robust counterpart method \cite{yang2008distributed,yang2014distributed,wang2016distributed} and the cutting-plane consensus algorithm \cite{burger2013polyhedral,burger2012distributed}. The former presents a direct strategy by transforming the DRCP into an equivalent and computationally tractable problem based on duality theory, which is then solved by an existing distributed algorithm \cite{yang2008distributed}. This method asymptotically converges to a feasible optimal solution, whereas it is primarily used for uncertain linear and conic quadratic programs (see the preface of \cite{ben2009robust}). Moreover, this method requires the global information of constraints and uncertainty sets, and thus, the system considered by this method is not fully distributed. The later is a fully distributed algorithm that relies on iteratively extending linear cutting-planes into the finite set of constraints to approximate the DRCP, thereby enabling the algorithm to asymptotically converge to the optimal solution of the DRCP. However, this polyhedral approximation framework for the feasible set of the DRCP results in an outer approximation \cite{burger2013polyhedral}, which does not guarantee the feasibility of the solution. To the best of our knowledge, none of the existing fully distributed algorithms guarantee the feasibility of the solution in the DRCP. Additionally, these algorithms for solving the DRCP fail to achieve finite-time termination, which is essential for practice. Therefore, the motivation of this article is {\it to propose a novel and fully distributed approach for finding a feasible and approximate optimal consensus solution for the DRCP within a finite number of iterations.} 
\par
From the aforementioned motivation, a novel cutting-surface consensus approach is proposed to solve the DRCP. The method is based on the constraint-handling strategies of the centralized cutting-surface algorithm \cite{mehrotra2014cutting}, and the centralized right-hand restriction algorithm \cite{Mitsos} which is used to deal with standard semi-infinite programs. Firstly, an approximated DRCP (ADRCP) is constructed by discretizing the semi-infinite constraints into a finite number of inequality constraints to retrieve computational tractability and restricting the right-hand side of the constraints with a positive parameter to ensure that a feasible solution for (DRCP) can be attained. Secondly, the ADRCP is solved by the distributed projected gradient algorithm, which is based on epigraphic reformulation \cite{xie2018distributed} and the gradient projected algorithm \cite{nedic}, and the resulting optimal solution is used to approximate that of the DRCP. Thirdly, a lower level program (LLP) is introduced to determine whether the obtained solution of the ADRCP is feasible to the original problem (DRCP). If infeasible, the ADRCP is tightened by adding a convex and possibly nonlinear cutting-surface based on the maximum constraint violation into the existing finite constraint set; otherwise, a feasible point of the DRCP is located and used to update the optimal solution candidate. Finally, if the stopping criterion is satisfied, the optimal solution candidate is taken as the desired approximate optimal solution of the DRCP; otherwise, the restriction parameters is reduced and then the routine is performed again with updated information based on the explored solution of the ADRCP until the desired solution is located. Based on the above procedures, a cutting-surface consensus approach has been formulated to converge to an optimal solution of the DRCP. The main contributions of this paper are summarized as follows. 
\begin{enumerate}
\item The DRCP with nonidentical constraints and nonidentical uncertainty sets is considered over a time-varying unbalanced directed network under the uniformly jointly strongly connected (UJSC) assumption since such a problem is more favorable in practical applications. To solve this problem, this paper proposes a novel and fully distributed cutting-surface consensus approach. Unlike the above-mentioned methods in \cite{yang2008distributed,yang2014distributed,wang2016distributed,burger2012distributed,burger2013polyhedral}, the proposed approach terminates finitely and yields a locally feasible solution for each agent satisfying consensus and zeroth-order stopping conditions to certain tolerances. 
\item A distributed projected gradient (DPG) algorithm is developed for solving the ADRCP problem. In addition, we prove its convergence and convergence rate of $O({\ln t}/{\sqrt{t}})$. Compared to existing subgradient projection algorithms \cite{nedic,weijianli2023,mai2019,Plin2016,Sliu2017,huaqiangli2019}, the proposed DPG algorithm can be applied to the distributed constrained optimization problems with nonidentical constraints under time-varying unbalanced directed network. 
\item Based on the zeroth-order stopping conditions \cite{Himmelblau1972} and finite-time consensus algorithm \cite{xie2017stop}, a distributed termination algorithm is developed and applied to the process of solving the ADRCP and the DRCP. Unlike the finite-time consensus technique \cite{xie2017stop,yadav2007distributed,manitara2016distributed} that only guarantees approximate consensus of the resulting solutions under strongly connected graphs, our algorithm can be used in UJSC networks and introduces zeroth-order stopping conditions to prevent the convergent distributed optimization algorithms from terminating prematurely on a very steep slope or a flat plateau. This ensures that all agents obtain approximate optimal solutions in the discrete-time setting.
\item It is mathematically proven that the cutting-surface consensus approach reaches termination in finite iterations and guarantees local feasibility of the resulting solutions.	
\end{enumerate}
\par
The remaining parts of the paper are organized as follows. Section 2 presents the distributed robust convex optimization problem and network topology. Section 3 presents our main result of locating a feasible solution of the DRCP that reaches consensus and zeroth-order stopping conditions to certain tolerances within finite iterations, and it also presents a mathematical proof of this result. Section 4 analyzes the effect of the algorithm with a numerical example and compares it with existing methods. Section 5 provides conclusions and an outlook on future work.
\par
Notation:
$\mathbb{R}^m$ denotes the set of all $m$ dimensional real column vector; $\mathbb{Z}$ denotes the set of all nonnegative integers; $x^\top$ and $A^\top$ denote the transpose of the vector $x$ and the matrix $A$; 
$I$ represents an identity matrix with compatible size; ${\bf 1}_m$ denotes the $m$-dimensional column vector with all entries equal to 1; $e_i$ is a column vector where the $i$-th element equals to 1 and the remaining elements equal to 0; $\Vert x \Vert$ represents the Euclidean norm of the vector $x$; $\vert a\vert$ denotes the absolute value of a scalar $a\in\mathbb{R}$; ${\rm{dist}}(\overline x,X)$ denotes the Euclidean distance from a vector $\overline x$ to a set $X$, i.e. ${\rm{dist}}(\overline x,X)={\inf}_{x\in X} \Vert \overline x-x\Vert$; $P_X[\overline x]$ denotes the Euclidean projection of a vector $\overline x$ on a closed convex set $X$, i.e. $P_X[\overline x]={\rm{arg}}\min_{x\in X}\Vert \overline x-x\Vert$.

\section{Problem Statement}
\label{sec:Problem Statement}
We consider a multi-agent system of $m$ agents in which the objective of all agents is to cooperatively solve a distributed robust convex program (DRCP) with nonidentical semi-infinite constraints of the form:
\begin{flalign} 
	\tag{${\rm DRCP}$}
	\label{DRCP}
	\begin{aligned}
		&\mathop{\min}\limits_{{x}\in X} &     & F(x)=\sum^m_{i=1}f_i({x})\\
		&\ \rm{s.t.}                             &     & g_i({x},y_i)\le0\quad  \forall y_i\in Y_i \quad i=1,2,\ldots,m,
	\end{aligned}
\end{flalign}
where ${x}$ is a common decision vector, $y_i$ is an uncertainty parameter of agent $i$, the global constraint set $X \subseteq \mathbb{R}^n$ is a nonempty and compact convex set, $Y_i \subseteq \mathbb{R}$ is a nonempty and compact convex set for any $i\in\{1, \ldots, m\}$, $f_i({x}):X\rightarrow \mathbb{R}$ is the local cost function of agent $i$, $F({x}):X\rightarrow \mathbb{R}$ is the global objective function, $g_i({x},y_i):X\times Y_i\rightarrow \mathbb{R}$ is the local constraint function of agent $i$. Assume $f_i$ and $g_i$ are convex with respect to ${x}$. Additionally, $g_i$ is assumed to be continuous on $X\times Y_i$ for any $i\!\in\!\{1, \ldots, m\}$. Set $x^*$ as an optimal solution of DRCP, and $x^*$ belongs to the optimal solution set $\mathcal{F}^*$.
\par
 Each agent only knows its local cost function $f_i$, constraint $g_i$, and uncertainty set $Y_i$, and this information can be different for each agent. It is defined that the local feasible set of agent $i$ is $\mathcal{F}_i=\left\{x\in X|g_i(x,y_i)\leq 0,\forall y_i\in Y_i\right\}$. To ensure the feasibility of (DRCP), we make the following assumption.
\par
\begin{assumption}[Feasibility of ${\rm (DRCP)}$ and Slater Point] \label{Assumption 1}	
$\mathcal{F}$ is a nonempty set, and 
there is a point $\hat x\in \mathcal{F}$ satisfying
$$g_i(\hat x,y_i)<0,\ \ \forall y_i\in Y_i, \ \ \forall i\in\{1,\ldots,m\},$$	
where $\mathcal{F}$ is the intersection of local feasible sets of all agents, i.e., $\mathcal{F}=\bigcap^m_{i=1}\mathcal{F}_i$.
\end{assumption}
\vspace{2mm}
\begin{remark}
Since an optimization problem with uncertain local cost functions can be transformed into an epigraph reformulation problem with the same form as the (\ref{DRCP}), it is without loss of generality that this paper only considers the uncertainty appearing in the local constraints.
\end{remark}
\vspace{2mm}
\par
The communication among agents is represented by time-varying directed graphs $\mathcal{G}(t)=(\mathcal{V},\mathcal{E}(t))$, where $\mathcal{V}=\left\{1,\ldots,m\right\}$ denotes the vertex set, and $\mathcal{E}(t)= \left\{1,\ldots,m\right\}^2$ is the edge set at time $t\in\mathbb{Z}$. A directed edge $(i,j)\in \mathcal{E}(t)$ indicates that agent $i$ can directly transmit information to agent $j$ at time $t$. $(i,i)$ represents the self-loop of agent $i$. $\mathcal{G}(t)$ is strongly connected if every agent is reachable by all other agents. For any pair of agents $(i,j) $ with $i,j\in \mathcal{V}$, the minimum number of edges between two agents is called the distance from $i$ to $j$, denoted as ${\rm dist}(i,j)$. The diameter of the graph is denoted as ${\rm diam}(\mathcal{G}(t))=\mathop{\max}\limits_{\forall i,j \in \mathcal{V}}\ {\rm dist}(i,j)$. At the time $t$, for agent $i$, its out-neighbors set is $N_i^{out}(t)=\left\{j|(i,j)\in \mathcal{E}(t)\right\}$, and its in-neighbors set is $N_i^{in}(t)=\left\{j|(j,i)\in \mathcal{E}(t)\right\}$. Define the weight matrix $A(t)=[a_{ij}(t)]\in \mathbb{R}^{m\times m}$, which satisfies that $a_{ij}(t)>0$ if $j\in N_i^{in}(t)$ and $a_{ij}(t)=0$, otherwise. In order to ensure that the information from any agent can be directly or indirectly transmitted to all other agents in finite time, the following assumption is made.
\par
\begin{assumption}[Uniformly Jointly Strongly Connectivity \cite{nedic}]\label{Assumption 2}
	The time-varying graphs $\left\{\mathcal{G}(t)\right\}$ are uniformly jointly strongly connected (UJSC), i.e. for all $t\ge0$, there is an integer $S>0$ making $\mathcal{G}(t:t+S)$ strongly connected, where
	$$\mathcal{G}(t:t+S)=(\mathcal{V},\mathcal{E}(t)\cup\ldots\cup \mathcal{E}(t+S-1)).$$
\end{assumption}
\par
The objective of our work is to propose an approach to solve the (\ref{DRCP}) within a finite number of iterations under time-varying unbalanced directed graphs while guaranteeing the local feasibility of the approximate optimal solution.
\section{Main results}
\label{sec:Approximation Problem and Finite Time Solving Method}
Since the constraints of the (\ref{DRCP}) are semi-infinite constraints, which means that an infinite number of inequality constraints are imposed on finite-dimensional decision variables, solving this problem is NP-hard. To solve the (\ref{DRCP}), this section firstly introduces an approximation problem of the (\ref{DRCP}) and proposes a distributed projected gradient (DPG) algorithm for solving this approximation problem under the setting of time-varying unbalanced directed networks. Then, we develop a distributed finite-time termination algorithm. By combining it with the DPG algorithm, all agents obtain approximate optimal and consensus solutions of the approximation problem within finite time. Next, we propose a distributed cutting-surface consensus algorithm that, by tuning the parameters of the approximation problem in each iteration, enables each agent to attain a locally feasible solution of the (\ref{DRCP}) satisfying consensus and zeroth-order stopping conditions to certain tolerances within a finite number of iterations. Finally, the main results on finite-time termination and local feasibility of the resulting solutions for the distributed cutting-surface consensus algorithm are established.
\subsection{Approximation Problem of DRCP and Solving Method}
\label{sec:Approximation Problem}
Motivated by the upper bounding problem in \cite{Mitsos}, an approximation problem of the (\ref{DRCP}) is introduced by discretizing the compact sets $Y_i$ into the finite sets $Y^k_i$ and restricting the right-hand sides of the constraints with positive parameters $\epsilon^k_i$ for all $i\in\mathcal{V}$. The mathematical description of the approximation problem is as follows.
\begin{flalign} 
	\tag{$\rm{ADRCP}^{\it k}$}
	\label{ADRCP}
	\hspace{-2mm}
	\begin{aligned}
		&\mathop{\min}\limits_{{x}\in X} &     & F(x)\!=\!\sum^m_{i=1}f_i({x})\\
		&\ \rm{s.t.}                             &     & g_i({x},y_i)\!\le\! -\epsilon^{k}_i\ \  \forall y_i\!\in\! Y^{k}_i \ \ i=1,\ldots,m,
	\end{aligned}
\end{flalign}
where for all $i\in\mathcal{V}$, $\epsilon^{k}_i$ is a restriction parameter of agent $i$, $\epsilon^{k}_i>0$, and $Y^{k}_i$ is a finite subset of $Y_i$, $Y^{k}_i \subset Y_i$. We set $x^{k,*}$ as an optimal solution of the (\ref{ADRCP}), and $x^{k,*}$ belongs to the optimal solution set $\mathcal{F}^{k,*}$.
\vspace{2mm}
\begin{remark}
In the subsequent distributed cutting-surface consensus algorithm, the (\ref{ADRCP}) denotes the distributed optimization problem constructed by the finite sets $Y^k_i$ and the restriction parameters $\epsilon_i^k$ of all agents at the iteration $k$. Here, $\epsilon^k_i$ represents the restriction parameter of the agent $i$ at the $k$-th iteration of the distributed cutting-surface consensus algorithm, and $Y_i^k\subset Y_i$ is a finite set containing values of the uncertainty parameter explored by agent $i$ within the previous $k$ iterations of this algorithm.
\end{remark}
\vspace{2mm}
\par
The problem (\ref{ADRCP}) is a standard distributed optimization problem with a finite number of local inequality constraints and a global constraint set over time-varying unbalanced directed graphs. In order to solve the (\ref{ADRCP}), we propose a distributed projected gradient  (DPG) algorithm under the UJSC assumption and prove the convergence and convergence rate of this algorithm. To show these results, we first introduce the following assumptions.

\begin{assumption}[Feasibility of ${\rm (ADRCP^k)}$ and Slater Point]\label{Assumption ip}
	Let $\mathcal{F}^k_i=\{x\in X| g_i(x,y_i)$ $ \le -\epsilon^k_i, \forall y_i\in Y^k_i\}$ be the local feasible set of agent $i\in\mathcal{V}$ in the (\ref{ADRCP}). Suppose that the intersection set $\mathcal{F}^k=\bigcap_{i=1}^m\mathcal{F}^k_i$ is nonempty, and there exists a point $\tilde x\in\mathcal{F}^k$ satisfying
	$$	g_i(\tilde x,y_i)<-\epsilon^k_i,\quad \forall y_i\in Y_i^k, \quad \forall i\in\{1,\ldots,m\}.$$	
\end{assumption}
\par
Under Assumption \ref{Assumption ip}, the (\ref{ADRCP}) has a nonempty optimal solution set, i.e., $\mathcal{F}^{k,*}\neq\emptyset$. 
\begin{assumption} \label{Assumption wr} The assumptions of row stochastic and weights rule are as follows:
	\begin{itemize}
		\item[(i)] The weight matrix $A(t)=[a_{ij}(t)]$ associated with $\mathcal{G}(t)$ is row stochastic, i.e. $\sum^m_{j=1} a_{ij}(t)=1$ for any $i\in \mathcal{V}$ and $t\in\mathbb{Z}$.
		\item[(ii)] There are self-loops in all $\mathcal{G}(t)$. Additionally, there exists a constant $0<\gamma<1$ such that $a_{ij}(t)\ge\gamma$ if $a_{ij}(t)>0$, and $a_{ii}(t)\ge\gamma$ for $\forall i,j\in\mathcal{V}$ and $\forall t\in\mathbb{Z}$.
	\end{itemize}
\end{assumption}
\par
Assumption \ref{Assumption wr}(i) establishes that each agent employs a convex combination of all agents' information in the distributed optimization algorithms. Assumption \ref{Assumption wr}(ii) states that all agents are assigned a uniform, non-zero lower bound on all weights corresponding to interactions with agents providing information, thus ensuring that the information from each agent influences the information of every other agent persistently in time \cite{nedic}.
\begin{assumption}[Stepsize Rule]\label{Assumption sr}
	The stepsize $\alpha(t)>0$ satisfies the following conditions: 
	\begin{equation}\nonumber
		\begin{aligned}
		&\sum^{\infty}_{t=0} \alpha(t)=\infty,\quad \sum^{\infty}_{t=0} \alpha^2(t)<\infty,\\
		&\alpha(t)\le\alpha(l)\qquad \forall t>l\ge0.
		\end{aligned}
	\end{equation}
\end{assumption}
\par
Assumption \ref{Assumption sr} ensures the convergence of the distributed optimization algorithms to the optimal point.
\par
Then, we start to design an optimization algorithm for solving the (\ref{ADRCP}). In order to coordinate with the subsequent distributed cutting-surface consensus algorithm, we consider two main aspects in the algorithm design. One is that the algorithm can be applied to time-varying unbalanced directed networks under the UJSC assumption, same as the cutting-surface consensus algorithm, which will not affect the applicability and practical value of the cutting-surface consensus algorithm. The other is that the algorithm is capable of solving the (\ref{ADRCP}) with nonidentical local constraints, and the resulting solution of each agent is feasible to its constraints. To the best of our knowledge, existing distributed optimization algorithms are unable to take both of these requirements into account (see Remark \ref{remark_dpg} and Table \ref{Table3} for details).
	\vspace{2mm}
	\setlength{\tabcolsep}{1mm}{
		\begin{table}[!t]
			\caption{Comparisons between the proposed DPG algorithm and the existing (sub)gradient projected algorithms}
			\label{Table3}
			\centering 
			\begin{tabular}{cccccc} 
				\toprule 
				\multirow{2}{*}{Literature} & Time & \multirow{2}{*}{Unbalanced} &\multirow{2}{*}{UJSC} &
				Nonidentical  & Convergence   \\ 
				&varying&&&constraints&property\\
				\midrule 
				\multirow{2}{*}{\!\!\cite{nedic}} & \checkmark & $\times$&\checkmark & $\times$&sum \\
				& $\times$ & $\times$ & $\times$&\checkmark&sum \\
				\cite{weijianli2023} & \checkmark & \checkmark & \checkmark&$\times$&weighted sum \\
				\cite{mai2019} & \checkmark & $\times$ & $\times$&\checkmark&sum \\
				\cite{Plin2016} & \checkmark & $\times$ & $\times$&\checkmark&sum \\  
				\cite{Sliu2017} & \checkmark & $\times$ & \checkmark&$\times$&sum \\
				\cite{huaqiangli2019} & \checkmark & \checkmark & \checkmark&$\times$&sum \\
				our work & \checkmark &\checkmark & \checkmark& \checkmark &sum \\
				\bottomrule 
			\end{tabular} 
		\end{table}
	} 
\begin{remark}[Dealing with Unbalance]\label{remark_dpg}
	The techniques for handling unbalance have been extensively studied. In the seminal work, \cite{kitsianos2012} proposed a reweighting technique to address fixed unbalanced networks, but the Perron vectors of the graph and global information were required. \cite{mai2019} developed an additional adaptive algorithm that estimates the Perron vectors, and then employs a reweighting technique without the need for global information. Nonetheless, this algorithm cannot be applied to time-varying graphs. \cite{nedic2015,congwang2023} proposed a push-sum consensus protocol to handle the unbalance of time-varying graphs, but this algorithm is complicated as it involves multiple nonlinear iterations. Recently, \cite{xie2018distributed} used an epigraph form of a constrained optimization problem to deal with unbalance, which is applicable to UJSC networks. Inspired by \cite{xie2018distributed}, we construct an epigraphic reformulation of the (\ref{ADRCP}) in Step (i) of our algorithm. Different from the D-RFP algorithm proposed in \cite{xie2018distributed}, we adopt the Euclidean projections to ensure that the agent estimates lie within their local feasible sets, ensuring the local feasibility of the solution instead of almost sure feasibility in \cite{xie2018distributed}.
\end{remark}
\vspace{2mm}
\par 
Therefore, we propose a new distributed projected gradient (DPG) algorithm based on the following two steps:
\par
{\bf Step (i):}  Proceed the following epigraphic reformulation for the (\ref{ADRCP}) to handle the unbalance of the distributed network:
\begin{equation} 	
	\hspace{-5mm}
	\label{epigraphic reformulation}
	\begin{aligned}
		&\mathop{\min}\limits_{(x,u)\in {\Theta}} & & \sum^m_{i=1}\frac{\mathbf{1}^\top_m  u}{m}\\
		&\quad \rm{s.t.}     & & f_i(x)-e^\top_i u\le0,\  i=1,2,...,m,     \\                        
		& & & g_i({x},y_i)\le -\epsilon^{k}_i,\ \forall y_i\in Y^{k}_i, \ i=1,2,...,m,		
	\end{aligned}
\end{equation}
where $\Theta\subseteq\mathbb{R}^{n+m}$ is the Cartesian product of $X$ and $\mathbb{R}^m$. Local objective function of agent $i$ is replaced with a linear function $f_0(u)=\frac{\mathbf{1}^\top_m  u}{m}$, while the local feasible set is  $\Omega_i=\left\{(x,u)\in \Theta|f_i(x)-e^{\top}_i u\le 0 \wedge  g_i(x,y_i)\le -\epsilon^k_i, \forall y_i\in Y^k_i\right\}$.
\par
To simplify the problem form of (\ref{epigraphic reformulation}), we consider the following equivalent form:
\begin{equation} 	
	\label{rp_adrcp}
	\begin{aligned}
		&\mathop{\min}\limits_{\theta\in {\Theta}} & & f_0(\theta)=c^\top 
		\theta\\
		&\ \ \rm{s.t.}     & & f_i(\theta)\le0,\quad \quad i=1,2,...,m      \\                        
		& & & g_i(\theta)\le 0,\quad \quad i=1,2,...,m		
	\end{aligned}
\end{equation}
where $\theta=(x,u)\in\Theta$, the local feasible set of agent $i$ is $\Omega_i=\left\{\theta\in\Theta|f_i(\theta)\le 0 \wedge g_i(\theta)\le 0\right\}$, and the global feasible set is $\Omega=\bigcap_{i=1}^m\Omega_i$.
\par
{\bf Step (ii):} Referring to the subgradient projection algorithm proposed in \cite{nedic}, each agent $i$ iteratively updates its local estimate $\theta_i(t+1)$ of the optimal solution to problem (\ref{rp_adrcp}) by
\begin{equation}
	\label{alg _dpg}
	\theta_i(t+1)=P_{\Omega_i}{\bigg[\sum_{j\in N^{in}_i(t)} a_{ij}(t)\theta_j(t)}-\alpha(t)c\bigg].
\end{equation}
\begin{remark}
	In the DPG algorithm, Euclidean projection is required for each iteration. The computational complexity of the projection is related to the projected set. If the projected set has a relatively simple structure, such as a half-space or an interval, the projected points can be obtained directly. Otherwise, a sub-optimization problem as shown below needs to be solved for each agent $i\in\mathcal{V}$ to find the projected point.
	\begin{equation} 	\nonumber
		\hspace{-11mm}
		\begin{aligned}
			&\mathop{\min}\limits_{\theta_i(t+1)\in {\Theta}} & & \Vert \theta_i(t+1)-\kappa_i(t)\Vert \\
			&\quad\  \rm{s.t.}     & & f_i(\theta_i(t+1))\le0   \\                        
			& & & g_i(\theta_i(t+1))\le 0,	
		\end{aligned}
	\end{equation}
	where $\kappa_i(t)=\sum_{j\in N^{in}_i(t)} a_{ij}(t)\theta_j(t)-\alpha(t)c$.
\end{remark}
\vspace{2mm}
\par
Finally, we prove the convergence and convergence rate of the DPG algorithm (see Appendix). Here are the results:
\begin{theorem}[Convergence of DPG Algorithm]\label{subproblem_convergence_theorem}
	Let $\theta^*$ be an optimal solution of the (\ref{ADRCP}), and $\theta^*$ belongs to the optimal solution set $\Omega^{*}$. Under Assumptions \ref{Assumption 2}--\ref{Assumption sr}, the distributed projected gradient algorithm reaches convergence to an optimal consensus solution, i.e., $\exists \theta^*\in\Omega^{*}:\ \lim_{t\rightarrow\infty}\theta_i(t)=\theta^*,\ \forall i\in\mathcal{V}$.
\end{theorem}
\begin{theorem}[Convergence Rate of DPG Algorithm]\label{subproblem_convergence_rate_theorem}
	Let Assumptions \ref{Assumption 2}--\ref{Assumption wr} be satisfied. If $\left\{\alpha(t)\right\}$ is nonincreasing and $\alpha(t)=O({1}/{\sqrt{t}})$, then for the distributed projected gradient algorithm, we have $\vert f_0(\hat\theta_i(t))-f_0^*\vert=O({\ln t}/{\sqrt{t}})$, where $\hat\theta_i(t)={\sum_{r=0}^{t}\alpha(r)\theta_i(r)}/{\sum_{\tau=0}^{t}\alpha(\tau)}$.
\end{theorem}

\subsection{Distributed Finite-Time Termination Algorithm}
\label{sec:Finite Time Termination}
In Section \ref{sec:Approximation Problem}, we developed a DPG algorithm to solve the (\ref{ADRCP}) and proved that the sequence $\{\hat \theta_i(t)\}_t$ of each agent converges at a sublinear rate towards $\theta^*$. However, in practical applications, we cannot make the algorithm iterate infinitely to obtain the optimal solution $\theta^*$. Therefore, it is essential to design an effective termination algorithm such that the DPG algorithm terminates after a finite number of time slots while all agents obtain approximate optimal solutions of the (\ref{ADRCP}). Currently, the most commonly used termination method in distributed optimization is to set maximum number of iterations, which is always a conservative consideration as it is affected by the network structures, the initial values of the decision variables, and the convergence rate of the recursive algorithm \cite{xie2017stop}. In \cite{yadav2007distributed,manitara2016distributed,xie2017stop}, finite-time consensus algorithms were proposed, equipped with the DPG algorithm, to obtain the approximate consensus solutions in finite time under strongly connected directed graphs. Nevertheless, this finite-time consensus algorithm is not sufficient for all agents to obtain approximate optimal solution. Specifically, when $\Vert \theta_i(t)-\theta_j(t)\Vert\le\epsilon$ holds for any $i,j\in\mathcal{V}$, there may exist the case that for any agent $i\in\mathcal{V}$, $\Vert \theta_i(t)-\theta^*\Vert\gg\epsilon$.
\par 
Therefore, based on finite-time consensus algorithm \cite{xie2017stop} and the zeroth-order stopping conditions \cite{Himmelblau1972}, we develop a novel finite-time termination algorithm for distributed consensus algorithms in the discrete-time setting, which ensures that all agents obtain approximate optimal consensus solutions. Moreover, our algorithm is applicable to time-varying unbalanced directed graphs under the UJSC assumption.
\par
Our aim is to make the approximation solution of each agent satisfy the consensus and zeroth-order stopping conditions to certain tolerances. The following is the stopping criterion: for any two agents $i\neq j$ and $i,j\in\mathcal{V}$
\begin{subequations}\label{stopping_criterion1}
			\begin{align}\label{stopping_criterion1_a}
			&\Vert \theta_i(t)-\theta_j(t)\Vert\le \epsilon_1 \\\label{stopping_criterion1_b}
			&\Vert \theta_i(t)-\theta_i(t-1)\Vert\le \epsilon_2 \\\label{stopping_criterion1_c}
			&\vert f_i(\theta_i(t))-f_i(\theta_i(t-1))\vert\le \epsilon_3,
		\end{align}
\end{subequations}
where $t\ge 1$. (\ref{stopping_criterion1_a}) refers to the consensus stopping condition, while (\ref{stopping_criterion1_b}) and (\ref{stopping_criterion1_c}) denote the zeroth-order stopping conditions based on the agent estimates $\theta_i$ and function values $f_i(\theta_i)$ respectively. Note that combining conditions (\ref{stopping_criterion1_b}) and (\ref{stopping_criterion1_c}) prevents the DPG algorithm from terminating prematurely on a very steep slope or a flat plateau.
\begin{definition}[Local $\epsilon_1$-consensus \cite{xie2017stop}]
	Given $\epsilon_1\!>\!0$, if the local estimate $\theta_i$ of agent $i\in\mathcal{V}$ satisfies $\mathop{\max}_{j\in N^{in}_i(t)\cup\left\{i\right\}}\Vert \theta_i(t)-\theta_j(t)\Vert\le\epsilon_1$, the agent $i$ is  deemed to reach local $\epsilon_1$-consensus.
\end{definition}
\par
Based on the stopping criterion (\ref{stopping_criterion1}), we propose a distributed finite-time termination algorithm. Firstly, each agent sends four-bit data to its out-neighbors between the time $t$ and $t+1$. The data is $[h^1_i(t), e^{1}_i(t), e^{2}_i(t), e^{3}_i(t)]$. The first bit $h^1_i(t)$ is used to judge whether the stopping criterion (\ref{stopping_criterion1}) is satisfied which will be shown later. The second bit $e^1_i(t)$ is used to count the latest number of times consecutively attaining local $\epsilon_1$-consensus at time $t$. The remaining two bits $e^2_i(t)$ and $e^3_i(t)$ record the latest number of times consecutively satisfying (\ref{stopping_criterion1_b}) and (\ref{stopping_criterion1_c}) for agent $i$ and its in-neighbors at time $t$.
\par
Then, each agent calculates (\ref{stoppingc1_1})--(\ref{stoppingc1_4}) at time $t+1$ according to its and its in-neighbors' information.
\begin{flalign}\label{stoppingc1_1}
	&h^1_i(t+1)\!=\!\mathop{\min}\limits_{j\in N^{in}_i(t)\cup\left\{i\right\}}\!\left\{h^1_j(t),e^{1}_j(t),e^{2}_j(t),e^{3}_j(t)\right\}\!+\!1,&
\end{flalign}
\begin{flalign}\label{stoppingc1_2}
&	e^{1}_i(t+1)=
	\begin{cases}
		e^{1}_i(t)\!+\!1,\quad \Vert \theta_i(t)-\theta_j(t)\Vert\le\epsilon_1\ \\ \ \qquad \qquad \quad\forall j\in{N_i^{in}(t)}\cup\left\{i\right\},\\
		0,\qquad \qquad \ otherwise,
	\end{cases}&
\end{flalign}
\begin{flalign}\label{stoppingc1_3}
&	e^{2}_i(t+1)=
	\begin{cases}
		e^{2}_i(t)\!+\!1,\quad 
		\Vert \theta_j(t)-\theta_j(t-1)\Vert\le\epsilon_2\ \\ \ \qquad \qquad \quad\forall j\in{N_i^{in}(t)}\cup\left\{i\right\},\\ 
		0,\qquad \qquad \ otherwise,
	\end{cases}&
\end{flalign}
\begin{flalign}\label{stoppingc1_4}
&	e^{3}_i(t+1)=
	\begin{cases}
		e^{3}_i(t)\!+\!1,\quad \vert f_j(\theta_j(t))-f_j(\theta_j(t-1))\vert\le\epsilon_3\\ \ \qquad \qquad \quad\forall j\in{N_i^{in}(t)}\cup\left\{i\right\},\\
		0,\qquad \qquad \ otherwise,
	\end{cases}&
\end{flalign}
where $h^1_i(1)\!=\!0$, $e^1_i(1)\!=\!0$, $e^2_i(1)\!=\!0$, $e^3_i(1)\!=\!0$ for all $i\in\mathcal{V}$.
\par
\begin{proposition}\label{stopping_criterion1_proposition}
	Let $D$ denote the diameter of the graph $\mathcal{G}(t:t+S)$, i.e., $D={\rm diam}(\mathcal{G}(t:t+S))$. Consider the distributed finite-time termination algorithm given by (\ref{stoppingc1_1})--(\ref{stoppingc1_4}) under Assumption \ref{Assumption 2}. If there exists an agent $i\in\mathcal{V}$ that satisfies $h^{1}_i(t)\ge SD+1$, then the stopping criterion (\ref{stopping_criterion1}) is satisfied for all agents of the multi-agent system at time $t$.
\end{proposition}
\begin{proof}
We use the proof idea in \cite{xie2017stop} to prove this proposition. Under Assumption 2, there exists a directed path $(j,j_1),(j_1,j_2),...,(j_q,i)$ from $j$ to $i$ with $q\le SD-1$ for any $i, j\in\mathcal{V}$. Since there is an agent $i\in\mathcal{V}$ satisfying $h^1_i(t)\ge SD+1$ at time t, then at time $t-1$, it follows that $h^1_{j_{q}}(t-1)\ge SD$, $e^1_{j_{q}}(t-1)\ge SD$, $e^2_{j_{q}}(t-1)\ge SD$, $e^3_{j_{q}}(t-1)\ge SD$.
 Similarly, when the time is $t-2$, it holds that $h^1_{j_{q-1}}(t-2)\ge SD-1$,
 $e^1_{j_{q-1}}(t-2)\ge SD-1$,
 $e^2_{j_{q-1}}(t-2)\ge SD-1$,
 $e^3_{j_{q-1}}(t-2)\ge SD-1$.
 Repeat the same steps, it follows that $h^1_{j}(t-q-1)\ge SD-q\ge 1$, $e^1_{j}(t-q-1)\ge SD-q\ge 1$, $e^2_{j}(t-q-1)\ge SD-q\ge 1$, $e^3_{j}(t-q-1)\ge SD-q\ge 1$.
\par Since the agent $j$ is arbitrary, the system has already satisfied the stopping criterion (\ref{stopping_criterion1}) at time $t-q-1$ and consistently satisfies (\ref{stopping_criterion1}) afterwards. Therefore, if there exists $i\in\mathcal{V}$ such that $h^{1}_i(t)\ge SD+1$, the network reaches the stopping criterion (\ref{stopping_criterion1}) at time $t$.
\end{proof}
\par
If the diameter $D$ is unknown, we can replace it with $m-1$, where $m$ is the number of the agents in the distributed network.
\par 
Based on the DPG algorithm and the finite-time termination algorithm, a distributed finite-time optimization method for solving (\ref{ADRCP}) in finite time is formulated. The pseudocode of the method is shown in Algorithm \ref{alg:3}. It is worth noting that in each iteration of Algorithm \ref{alg:3}, each agent only exchanges its local estimates $\theta_i(t)$, $\theta_i(t-1)$, the corresponding function values $f_i(\theta_i(t))$, $f_i(\theta_i(t-1))$, and the termination algorithm parameters $h^{1}_i(t)$, $e^{1}_i(t)$, $e^{2}_i(t)$, $e^{3}_i(t)$ with its neighboring agents, while the local constraint functions, local cost functions, and other private information remain undisclosed, thus protecting the privacy of each agent.
\begin{algorithm}[!t]
	\renewcommand{\algorithmicrequire}{\textbf{Initialization:}}
	\renewcommand{\algorithmicensure}{\textbf{Input:}}
	\caption{Distributed projected gradient algorithm with finite-time termination for solving the {(\ref{ADRCP})}}
	\label{alg:3}
	\begin{algorithmic}[1]
		\REQUIRE  For all $i\in\mathcal{V}$, $
		\Vert \theta_i(0)\Vert=+{\rm Inf}$, $
		\Vert \theta_i(1)\Vert=0$, $h^1_i(1)=0$, $e^1_i(1)=0$, $e^2_i(1)=0$, $e^3_i(1)=0$; $t=1$.
		\ENSURE  $S$; $D$; $\epsilon_1$, $\epsilon_2$, $\epsilon_3$; $\alpha(t)$.
		\renewcommand{\algorithmicensure}{\textbf{Repeat:}}
		\ENSURE 
		\STATE Send $\theta_i(t\!-\!1)$, $\theta_i(t)$, $f_i(\theta_i(t\!-\!1))$, $f_i(\theta_i(t))$, $h^{1}_i(t)$, $e^{1}_i(t)$, $e^{2}_i(t)$ and $e^{3}_i(t)$ to the out-neighbors of agent $i$.\\
		\STATE  Update $\theta_i(t+1)$ according to (\ref{alg _dpg}), and update $h^{1}_i(t+1)$, $e^{1}_i(t+1)$, $e^{2}_i(t+1)$ and $e^{3}_i(t+1)$ according to (\ref{stoppingc1_1})--(\ref{stoppingc1_4}).	
		\STATE Set $t\leftarrow t+1 $		
		\IF {$h^{1}_i(t)\ge SD+1$}		 
		\STATE $\bullet$ $\left[{x_i^k}^\top, {u_i^k}^\top\right]^\top\leftarrow\theta_i(t)$
		\STATE $\bullet$ return $x^k_i$, {\bf exit}
		\ENDIF
	\end{algorithmic}	
\end{algorithm}
\begin{corollary}\label{corollary_dftt}
Under Assumptions \ref{Assumption 2}--\ref{Assumption sr}, Algorithm \ref{alg:3} terminates finitely and each agent obtains a local feasible solution $x^k_i$ of (\ref{ADRCP}) satisfying consensus and zeroth-order stopping conditions to certain tolerances.
\end{corollary}
\begin{proof}
According to Theorem 1 and the triangle inequality, for any $i\neq j\in\mathcal{V}$, it follows that $$\mathop{\lim}\limits_{t\rightarrow\infty}\Vert\theta_i(t)-\theta_j(t)\Vert\!\le\! \mathop{\lim}\limits_{t\rightarrow\infty}[\Vert\theta_i(t)-\theta^*\Vert\!+\!\Vert\theta_j(t)-\theta^*\Vert]\!=\!0.$$ Due to the positivity of the norm, it can be obtained that $\mathop{\lim}\limits_{t\rightarrow\infty}\Vert\theta_i(t)-\theta_j(t)\Vert=0.$ Therefore, for any $\epsilon_1>0$, 
$$\exists T_1:\ \Vert \theta_i(T_1)-\theta_j(T_1)\Vert\le \epsilon_1,\ \forall i,j\in\mathcal{V}.$$ Similarly, we can obtain that for any $\epsilon_2, \epsilon_3>0$, 
$$\exists T_2:\ \Vert \theta_i(T_2)-\theta_i(T_2-1)\Vert\le \epsilon_2,\ \forall i\in\mathcal{V},$$
$$\exists T_3:\ \vert f_i(\theta_i(T_3))-f_i(\theta_i(T_3-1))\vert\le \epsilon_3,\ \forall i\in\mathcal{V}.$$
Therefore, we can conclude that after at most $\max\{T_1, T_2, T_3\}$ iterations of Algorithm \ref{alg:3}, each agent stops updating and furnishes a point $x^k_i$ satisfying consensus and zeroth-order stopping conditions within certain tolerances. In addition, since the Euclidean projection operators are adopted, $x^k_i$ lies within its local feasible region $\mathcal{F}_i^k$.
\end{proof}

\subsection{Distributed Cutting-Surface Consensus Approach for Solving The DRCP}
\label{sec:Algorithm Development and Analysis}
In this subsection, we propose a distributed cutting-surface consensus approach for solving the (\ref{DRCP}), which contains three nested selection structures. The pseudocode and graphic illustration of the algorithm are shown in Algorithm \ref{alg:DRCP} and Figure \ref{graphic_illustration}, respectively. 

\begin{algorithm}[h]
	\renewcommand{\algorithmicrequire}{\textbf{Input:}}
	\renewcommand{\algorithmicensure}{\textbf{Iteration:} For each agent $i\in\mathcal{V}$,}	
	\caption{Distributed cutting-surface consensus algorithm}
	\label{alg:DRCP}
	\begin{algorithmic}[1]
		\REQUIRE For all $i\in\mathcal{V}$, restriction parameter $\epsilon^0_i>0$, finite or empty set $Y^0_i\subset Y_i$, optimal solution candidate $z^0_i$, $\Vert z_i^0\Vert=+{\rm Inf}$; iteration counter $k=0$; reduction parameter $r>1$.
		\ENSURE
		\STATE Compute $x^k_i$ by solving (\ref{ADRCP}) using Algorithm \ref{alg:3}.
		\IF {$x^k_i\in\emptyset$}
		\renewcommand{\algorithmicensure}{\quad \ \ {\color{gray}\% (${\rm ADRCP}^k$) is infeasible (Case I)}}	
	   \ENSURE
		\STATE $\bullet$ $\epsilon^{k+1}_i\leftarrow\epsilon^k_i/r$
		\STATE $\bullet$ $Y^{k+1}_i\leftarrow Y^k_i$
		\ELSE
		\renewcommand{\algorithmicensure}{\quad \ \ {\color{gray}\% (${\rm ADRCP}^k$) is feasible}}	
		\STATE $\bullet$ {$y^{\max,k}_i\in \arg \mathop{\max}\limits_{{y_i}\in Y_i} g_i(x^k_i,y_i)$\qquad\qquad\qquad\quad$({\rm LLP}_i)$}\\
		{\color{gray}\% Some effective numerical methods for solving $({\rm LLP}_i)$ are discussed in Remark \ref{remark_5}.}
		\STATE $\bullet$ {$g^{\max,k}_i=g_i(x^k_i,y^{\max,k}_i)$}
		\IF {$g_i^{\max,k}> 0$}	
		\renewcommand{\algorithmicensure}{\qquad\ \ {\color{gray}\% $x^k_i$ is infeasible for (DRCP) (Case II)}}	
		\ENSURE
	\STATE $\ast$ $\epsilon^{k+1}_i\leftarrow\epsilon^k_i$
	\STATE $\ast$ $Y^{k+1}_i\leftarrow Y^k_i\cup\{y^{\max,k}_i\}$
	\renewcommand{\algorithmicensure}{\qquad\quad\ \  {\color{gray}\% add feasibility cuts}}	
	\ENSURE
	\STATE $\ast$ $z_i^{k+1}\leftarrow z_i^{k}$
	\ELSE 
	\renewcommand{\algorithmicensure}{\qquad\ \ {\color{gray}\% $x^k_i$ is feasible for (DRCP)}}	
		\STATE $\ast$ $z_i^{k+1}\leftarrow x^k_i$.
		\STATE $\ast$ {Check whether stopping criterion (\ref{stopping_criteria_2}) is met by using Algorithm \ref{alg:2}.}
		\IF {not}
		\renewcommand{\algorithmicensure}{\qquad\quad\ \  {\color{gray}\% no desirable solution obtained (Case III)}}
		\ENSURE
		\STATE $\star$  $\epsilon^{k+1}_i\leftarrow\epsilon^k_i/r$
		\STATE $\star$  $Y^{k+1}_i\leftarrow Y^k_i$
		\renewcommand{\algorithmicensure}{\qquad\qquad\ \  {\color{gray}\% add optimality cuts}}	
		\ENSURE
		\ELSE
		\STATE  
		$\star$ return $z_i^{k+1}$ and \textbf{Terminate.}
		\ENDIF		
		\ENDIF
		\ENDIF
		\STATE Set $k\leftarrow k+1$
	\end{algorithmic}
\end{algorithm} 
\begin{figure}[!t]
	\centerline{\includegraphics[width=\columnwidth]{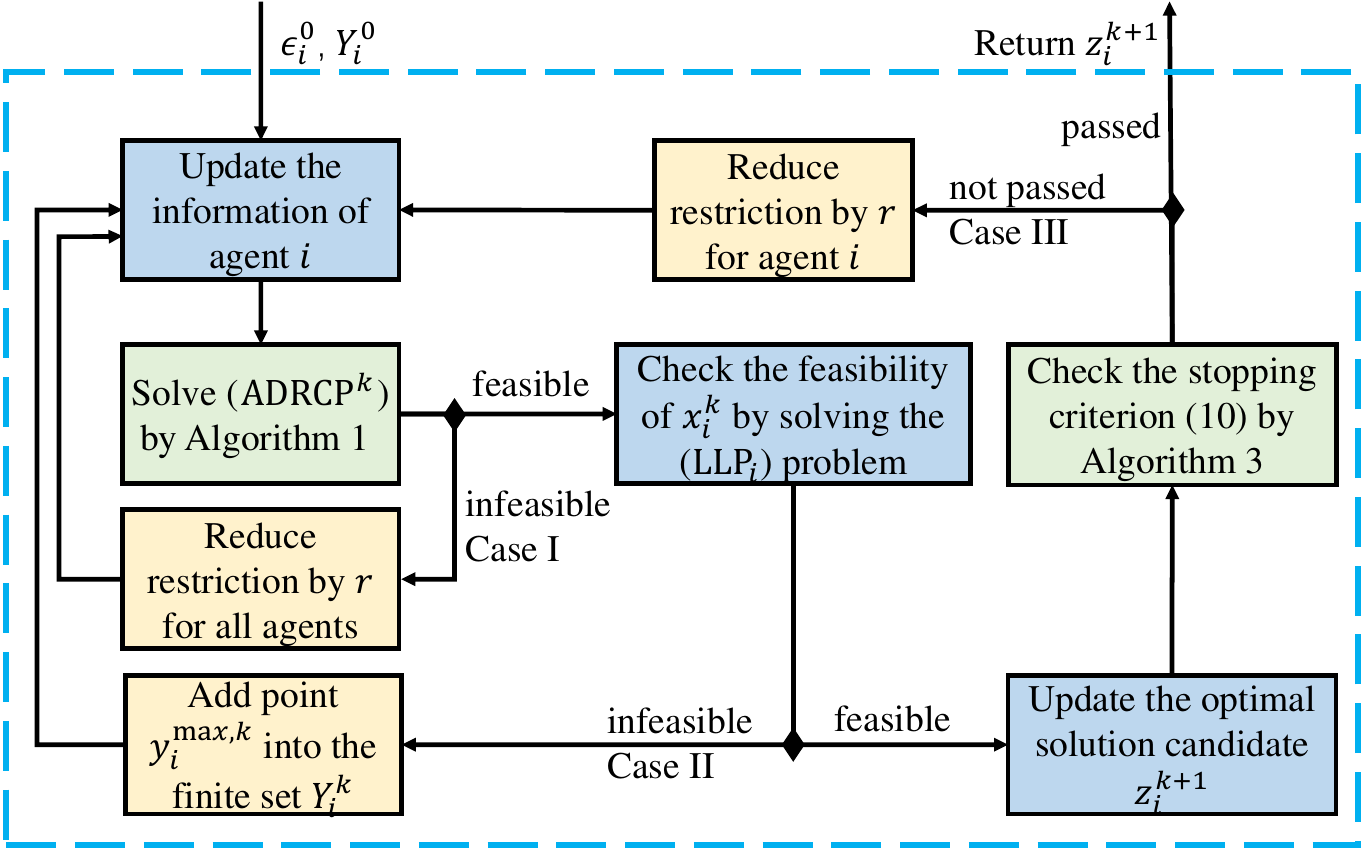}}
	\caption{Graphic illustration of Algorithm \ref{alg:DRCP}. The steps in yellow and blue only require computation, information update and storage within the agent. The green steps require information exchange between each agent on this basis.}
	\label{graphic_illustration}
\end{figure}
\subsubsection{The first selection structure}
 The (\ref{ADRCP}) is neither a relaxation nor a restriction of the (\ref{DRCP}) in general. In the input step, we choose any finite or empty set $Y_i^0\subset Y_i$ and any initial restriction parameter $\epsilon^0_i>0$ for each agent.
  Since there are relatively large restriction parameters $\epsilon^k_i$ in the first few iterations of Algorithm \ref{alg:DRCP}, the global feasible set $\mathcal{F}^k$ of the (\ref{ADRCP}) may be an empty set, in which case, specifically, there exist two situations for the (\ref{ADRCP}): one is that the local feasible set of a certain agent is empty, and the other is that the local feasible sets of all agents are non-empty while their intersection is empty. We use Algorithm \ref{alg:3} to solve the (\ref{ADRCP}). For the first situation, each agent is forced to terminate and cannot obtain a solution. For the second situation, all agents update their local estimates infinitely since the consensus stopping condition (\ref{stopping_criterion1_a}) fails to be satisfied.
\par
In order to terminate the DPG algorithm for any combinations of $Y^k_i$ and $\epsilon^k_i$ in finite time without Assumption \ref{Assumption ip}, we set maximum iteration time $T$ on the basis of Algorithm \ref{alg:3}. When the DPG algorithm cannot yield a solution or fail to reach the stopping criterion (\ref{stopping_criterion1}) within $T$, we can determine that the (\ref{ADRCP}) is infeasible. Therefore, we relax the (\ref{ADRCP}) by proportionally decreasing the restriction parameters $\epsilon_i^k$ for all agents. If the DPG algorithm terminates by satisfying the stopping criterion (\ref{stopping_criterion1}) within $T$, then we obtain the approximate optimal solution of the (\ref{ADRCP}) and update $x_i^k$.
\subsubsection{The second selection structure} 
When the (\ref{ADRCP}) is feasible, the obtained solution $x^k_i$ of agent $i$ has two cases. One is that $x^k_i$ is in its local feasible region of the (\ref{DRCP}), i.e., $x^k_i\!\in\!\mathcal{F}_i$, the other is that $x^k_i$ is not in agent $i$'s feasible region of the (\ref{DRCP}), i.e., $x^k_i\!\notin\!\mathcal{F}_i$. To judge whether $x^k_i$ belongs to $\mathcal{F}_i$, we introduce a lower level program $({\rm LLP}_i)$.
\par
If $g^{\max,k}_i\le0$, it indicates that $x^k_i\in \mathcal{F}_i$. Then, we assign the value of $x_i^k$ to the (\ref{DRCP}) optimal solution candidate $z_i^{k+1}$, and reduce $\epsilon^k_i$ proportionally to relax the (\ref{ADRCP}) problem which is aimed to obtain a solution that is closer to the optimal solution point of the (\ref{DRCP}) in later iterations ({\it Optimality cuts}). If $x_i^k\notin \mathcal{F}_i$, the obtained $x_i^k$ has no practical significance and needs to be discarded. At this time, we add the point $y^{\max,k}_i$ to $Y_i^k$ to tighten the (\ref{ADRCP}), aiming to make $x^k_i$ locally feasible to the (\ref{DRCP}) in later iterations ({\it Feasibility cuts}).
\par
\vspace{2mm}
\begin{remark}\label{remark_5}
Given that $x_i^k$ is fixed for each agent $i\in\mathcal{V}$, the $({\rm LLP}_i)$ is a standard optimization problem with a set constraint $y_i\in Y_i$, where $y_i$ is the decision variable. In order to ensure the local feasibility of the solutions generated by Algorithm \ref{alg:DRCP}, each agent needs to solve its $({\rm LLP}_i)$ at {\it global} optimality by local computation. \cite{burger2012distributed,mutapcic2009cutting} used convex numerical methods to locate the maximum violation for the convex $({\rm LLP}_i)$. When the $({\rm LLP}_i)$ is nonconvex, \cite{FU2015184} presents an effective numerical method under the assumption that $g_i$ is continuously differentiable. Specifically, we first introduce an additional function $\chi(y_i)$:
\begin{flalign}
		\dot \chi(y_i)=
	\begin{cases}
	0,\qquad\qquad\ \  \chi(y_i)\ge g_i(x^k_i,y_i)\\ \qquad\qquad\quad\ \ \ or\ \dot g_i(x^k_i,y_i)\le 0,\\
		\dot g_i(x^k_i,y_i), \quad\ otherwise,
	\end{cases}
\end{flalign}
where $x^k_i$ is fixed, $y_i\in Y_i$, $\dot \chi(y_i)$ represents the derivative of the function $\chi(y_i)$, and $\dot g_i(x^k_i,y_i)$ represents the derivative of the function $g_i(x^k_i,y_i)$ with respect to $y_i$. Set the compact set $Y_i=\left\{y_i|y_0 \le y_i\le y_f\right\}$ and $\chi(y_0)=g_i(x^k_i,y_0)$. Then, adopting an adaptive step-size solver with rigorous state-event location, we obtain the maximum constraint violation ${\max}_{y_i\in Y_i}\ g_i(x^k_i,y_i)=\chi(y_f)$. Moreover, $y^{\max,k}_i$ is obtained from the last event triggered in the integration process.
\end{remark}
\vspace{2mm}
\subsubsection{The third selection structure}
To ensure finite-time termination of Algorithm \ref{alg:DRCP}, we adapt the distributed finite-time termination algorithm in Section \ref{sec:Finite Time Termination} to verify the following stopping criterion:
for any two agents $i\neq j$, and $i,j\in\mathcal{V}$
\begin{subequations}\label{stopping_criteria_2}
	\begin{align}\label{stopping_criteria_2(a)}
		&\Vert z_i^{k+1}-z_j^{k+1}\Vert\le \epsilon_4 \\\label{stopping_criteria_2(b)}
		&\Vert z_i^{k+1}-z_i^{k}\Vert\le \epsilon_5 \\\label{stopping_criteria_2(c)}
		&\vert f_i(z_i^{k+1})-f_i(z_i^{k})\vert\le \epsilon_6
	\end{align}
\end{subequations}
where (\ref{stopping_criteria_2(a)}) refers to the consensus stopping condition, while (\ref{stopping_criteria_2(b)}) and (\ref{stopping_criteria_2(c)}) denote the zeroth-order stopping conditions based on the optimal solution candidates $z_i^k$ and function values $f_i(z_i^k)$ respectively. The pseudocode is shown in Algorithm {\ref{alg:2}}.
\begin{algorithm}[!t]
	\renewcommand{\algorithmicrequire}{\textbf{Initialization:}}
	\renewcommand{\algorithmicensure}{\textbf{Input:}}
	\caption{Distributed termination algorithm for cutting-surface consensus approach}
	\label{alg:2}
	\begin{algorithmic}[1]
		\REQUIRE $h^2_i(t)=0$, $e^4_i(t)=0$, $e^5_i(t)=0$, $e^6_i(t)=0$, $\forall i\in\mathcal{V}$; $t=0$.	
		\ENSURE $z_i^k$, $z_i^{k+1}$, $\forall i\in\mathcal{V}$; $S$; $D$; $\epsilon_4$, $\epsilon_5$, $\epsilon_6$.
		\renewcommand{\algorithmicensure}{\textbf{Repeat:}}
		\ENSURE 	
		\STATE  {\bf Local information exchange:} send the values $z_i^k$, $z_i^{k+1}$, $f_i(z_i^k)$, $f_i(z_i^{k+1})$, $e^{4}_i(t)$, $e^{5}_i(t)$, $e^{6}_i(t)$ and $h^{2}_i(t)$ to the out-neighbors of agent $i$\\
		\STATE  {\bf Local variable calculate:} 
		\begin{flalign}\label{stoppingc2_1}
			&h^2_i(t\!+\!1)\!=\!\mathop{\min}\limits_{j\in N^{in}_i(t)\cup\left\{i\right\}}\!\!\left\{h^2_i(t),e^{4}_i(t),e^{5}_i(t),e^{6}_i(t)\right\}\!+\!1,&
		\end{flalign}
		\begin{flalign}\label{stoppingc2_2}
			&	e^{4}_i(t\!+\!1)\!=\!
			\begin{cases}
				e^{4}_i(t)\!+\!1,\quad \Vert z_i^{k+1}-z_j^{k+1}\Vert\!\le\!\epsilon_4\ \\ \qquad \qquad \quad \ \forall j\in{N_i^{in}(t)}\cup\left\{i\right\},\\
				0,\quad\ \qquad \quad otherwise,
			\end{cases}&
		\end{flalign}
		\begin{flalign}\label{stoppingc2_3}
			&	e^{5}_i(t\!+\!1)\!=\!
			\begin{cases}
				e^{5}_i(t)\!+\!1,\quad
				\Vert z_j^{k+1}-z_j^{k}\Vert\!\le\!\epsilon_5\ \\ \qquad \qquad \quad\ \forall j\in{N_i^{in}(t)}\cup\left\{i\right\},\\ 
				0,\qquad\ \quad \quad otherwise,
			\end{cases}&
		\end{flalign}
		\begin{flalign}\label{stoppingc2_4}
			&	e^{6}_i(t\!+\!1)\!=\!
			\begin{cases}
				e^{6}_i(t)\!+\!1,\ \vert f_j(\!z_j^{k+1}\!)\!-\!f_j(z_j^{k})\vert\!\le\!\epsilon_6\\ \qquad \qquad \ \ \forall j\in{N_i^{in}(t)}\cup\left\{i\right\},\\
				0,\qquad\ \quad \ otherwise.
			\end{cases}&
		\end{flalign}
		
		\STATE Set $t\leftarrow t+1 $		
		\IF {$t>SD$\ \&\ $h^{2}_i(t)\ge SD+1$}		 
		\STATE stopping criterion (\ref{stopping_criteria_2}) is satisfied, {\bf exit}
		\ENDIF
		\renewcommand{\algorithmicensure}{\textbf{Until:}}
		\ENSURE {a predefined stopping rule (e.g.  $t>SD+1$) is satisfied.}		
	\end{algorithmic}	
\end{algorithm}

\begin{proposition}\label{stopping_criterion2_proposition}
	Consider Algorithm \ref{alg:2} under Assumption \ref{Assumption 2}. If there exists an agent $i\in\mathcal{V}$ that satisfies $h^{2}_i(t)\ge SD+1$ at time $SD+1$, then the stopping criterion (\ref{stopping_criteria_2}) is satisfied for all agents of the multi-agent system.
\end{proposition}
\begin{proof}
	Similar to the proof idea of Proposition \ref{stopping_criterion1_proposition}, the above result can be easily derived.
\end{proof}
\par
Algorithm \ref{alg:2} provides Algorithm \ref{alg:DRCP} with termination detection. If there exists an agent satisfying $g_i^{\max,k}\le 0$ in Step 13 of Algorithm \ref{alg:DRCP}, Algorithm \ref{alg:2} will start to execute. When the network satisfies the stopping criterion (\ref{stopping_criteria_2}), all agents spontaneously issue an exit command to terminate Algorithm \ref{alg:DRCP}. Additionally, an approximate optimal solution $z_i^{k+1}$ of the (\ref{DRCP}) is output for each agent.
\par
Different from the distributed finite-time termination algorithm introduced in Section \ref{sec:Finite Time Termination}, Algorithm \ref{alg:2} serves primarily to judge the satisfaction of the stopping criterion (\ref{stopping_criteria_2}) of the optimal solution candidates $z_i^k$ for all agents at the $k$-th iteration, where the value of $z_i^k$ does not change with respect to time $t$, but it is affected by the iteration $k$. In addition, Algorithm \ref{alg:2} requires only a fixed number of time ($SD+1$) to determine whether Algorithm \ref{alg:DRCP} needs to be terminated.
\color{black}
\subsection{Convergence Analysis}
This subsection presents the main result of Algorithm \ref{alg:DRCP}. Firstly, to ensure finite-time termination of Algorithm \ref{alg:DRCP}, we make the following assumption.
\begin{assumption}[Strict convexity]\label{Assumption 5}
	The global objective function $F(x)$ of the (\ref{DRCP}) is strictly convex, i.e., for any two point $x_1\neq x_2\in X$, $0<\beta<1$, there is
	\begin{equation}\nonumber
		F(\beta x_1+(1-\beta) x_2)<\beta F(x_1)+(1-\beta)F(x_2).
	\end{equation} 
\end{assumption}
\par
 Note that for the case that $F(x)$ is a convex function, we can adopt a tie-break rule in solving the (\ref{ADRCP}) to ensure the uniqueness of the solution.
\begin{proposition}[Finite-time Termination]\label{dcsc_ftt_proposition}
	For each agent $i\in\mathcal{V}$, take any initial restriction parameter $\epsilon_i^0>0$, any initial finite or empty set $Y^0_i\subset Y_i$, and any reduction parameter $r> 1$. Under Assumptions \ref{Assumption 1}, \ref{Assumption 2}, \ref{Assumption wr}--\ref{Assumption 5}, the distributed cutting-surface consensus algorithm terminates in a finite number of iterations and obtain solutions for all agents that satisfy consensus and zeroth-order stopping conditions to certain tolerances.
\end{proposition}
\begin{proof}
To prove the finite-time termination of the algorithm, we only need to exclude three possible infinite cases (Case I--III), which are presented in Figure \ref{graphic_illustration}.\\
{\bf 1) Exclusion of infinite Case I:}
By Assumption 1, there exists at least one slater point $\hat{x}$ for the (${\rm DRCP}$), we have
\begin{equation}\label{100000000002}
	\hat{x}\in\hat{\mathcal{F}}= \bigcap^m_{i=1} \left\{x\in X|g_i(x,y_i)\le-\epsilon_i^s,\forall y_i\in Y_i\right\},
\end{equation}
where $0<\epsilon_i^s\le \min_{y_i\in Y_i}{-g_i(\hat x,y_i)},\ \forall i\in\mathcal{V}$.
Assume that the (${\rm ADRCP}^k$) is infeasible at the $k$-th iteration, i.e., $\mathcal{F}^k=\emptyset$. Then, after at most $K_1=\max_{i\in\mathcal{V}} \lceil\log_r \epsilon^k_i/\epsilon^s_i \rceil$ updates of restriction parameters $\epsilon^k_i$, it follows that $0<\epsilon_i^{K_1}\le\epsilon_i^s$ and $Y_i^{K_1}\subset Y_i$. Hence, it can be obtained that $\mathcal{F}^{K_1}=\bigcap^m_{i=1}\{x\in X|g_i(x,y_i)\le -\epsilon^{K_1}_i, \ \forall y_i\in Y_i^{K_1}\}\supset\hat{\mathcal{F}}\neq\emptyset$. Therefore, after at most $K_1$ executions of Case I in Algorithm \ref{alg:DRCP}, the (${\rm ADRCP^k}$) becomes feasible.\\
{\bf 2) Exclusion of infinite Case II:}
Referring to the proof idea of Lemma 2.4 in \cite{Mitsos}, we exclude the infinite Case II for any agent in the loop of Algorithm \ref{alg:DRCP}. For any $i\in\mathcal{V}$, we first consider a sequence of solutions of the (${\rm ADRCP^k}$) for this agent. Since $X$ is a compact set, we can select a converging subsequence $x^k_i\rightarrow \widetilde{x}_i$ for agent $i$, where $\widetilde{x}_i$ is the convergence point of the subsequence.
Assume that at the $k$-th iteration, the point $x^k_i$ does not satisfy its local constraints $\mathcal{F}_i$, and $ y^{\max,k}_i$ is obtained through solving the (${{\rm LLP}_i}$). Thus, the following relation holds:
$$g_i(x^k_i, y^{\max,k}_i)>0.$$
In step 10 of Algorithm \ref{alg:DRCP}, the point $y^{\max,k}_i$ is added into the finite set of agent $i$. Hence, for any $l>k\ge 0$, it follows that $y^{\max,k}_i\in Y^l_i$.
In Section \ref{sec:Approximation Problem}, the DPG algorithm guarantees the local feasibility of the solution due to the adoption of the Euclidean projection. Hence, for any $l>k\ge 0$, it holds that $$g_i(x^l_i, y^{\max,k}_i)\le -\epsilon^l_i.$$
Due to the compactness of $X\times Y_i$ and continuity of $g_i(x,y_i)$, $g_i(x,y_i)$ is uniformly continuous and hence
\begin{equation}\label{10000000003}
\exists \delta>0\ \  g_i(x, y^{\max,k}_i)\le -\epsilon^l_i/2<0\ \ \Vert x-x^l_i\Vert<\delta\ \ \forall  l>k\ge 0.
\end{equation}
As a result of $x^k_i\rightarrow \widetilde{x}_i$, for any $\delta>0$ there is 
\begin{equation}\label{10000000004}
\exists K_2:\ \Vert x^l_i-x^k_i\Vert<\delta\ \ \forall l,k:\  l>k\ge K_2.
\end{equation}
Combine the results (\ref{10000000003}) and (\ref{10000000004}), we can obtain that
\begin{equation}\nonumber
	\exists K_2: \  g_i(x^k_i, y^{\max,k}_i)\le -\epsilon_i^l/2<0\ \ \forall k:\  k\ge K_2.
\end{equation}
Therefore, we can conclude that for any agent $i\in\mathcal{V}$, after at most $K_2$ iterations of Algorithm \ref{alg:DRCP}, the point $x^{K_2}_i$ obtained by solving the (${{\rm ADRCP}^{K_2}}$) lies within its local constraints set $\mathcal{F}_i$ in the $({\rm DRCP})$.\\
{\bf 3) Exclusion of infinite Case III:}
	Suppose that the (${{\rm ADRCP}^k}$) is feasible at the $k$-th iteration of Algorithm \ref{alg:DRCP}. By using Algorithm \ref{alg:3}, each agent can obtain an approximate optimal solution $x^k_i$ of the (${\rm ADRCP}^k$) (cf. Corollary \ref{corollary_dftt}). Moreover, for any agent $i\in\mathcal{V}$, we have
	\begin{equation}
		\Vert x^k_i-x^{k,*}\Vert=\delta^k_i
	\end{equation}
    where $k\ge 0$. $\delta^k_i$ is a positive scalar satisfying $0\le\delta^k_i\ll \min\left\{\epsilon_4,\epsilon_5,\epsilon_6\right\}$ as Algorithm \ref{alg:3} is fully iterated by setting small parameters $\epsilon_1,\epsilon_2,\epsilon_3$, where $\max\left\{\epsilon_1,\epsilon_2,\epsilon_3\right\}\ll \min\left\{\epsilon_4,\epsilon_5,\epsilon_6\right\}$. 
    Then, we have $x^k_i=x^{k,*}+\delta^k_iv^k_i$, where $v^k_i=({x^k_i-x^{k,*}})/\Vert x^k_i-x^{k,*}\Vert$. For any $i\neq j\in\mathcal{V}$, suppose that they update $z_i^{k_i+1}, z_j^{k_j+1}$ by Step 13 at the $k_i$-th and $k_j$-th iterations of Algorithm \ref{alg:DRCP}, respectively. Thus, we have
    \begin{equation}\label{10000000005}
    	\hspace{-3mm}
    	\begin{aligned}
    		\Vert z_i^{k_i+1}\!-\!z_j^{k_j+1}\Vert&=\!\Vert x_i^{k_i}\!-\!x_j^{k_j}\Vert\\
    		&\le\! \Vert x^{k_i,*}\!-\!x^{k_j,*}\Vert\!+\!\Vert\delta^{k_i}_i v^{k_i}_i\!-\!\delta^{k_j}_j v^{k_j}_j\Vert\\
    		&\le\! \Vert x^{k_i,*}\!-\!x^{k_j,*}\Vert\!+\!\delta^{k_i}_i\!+\!\delta^{k_j}_j,\\
    		&\le\! \Vert x^{k_i,*}\!-\!x^{*}\Vert\!+\!\Vert x^{k_j,*}\!-\!x^{*}\Vert\!+\!\delta^{k_i}_i\!+\!\delta^{k_j}_j.
    	\end{aligned}
    \end{equation} 
    where the last three inequalities hold based on the triangle inequality. Next, we prove the convergence of the sequence $\{x^{k,*}\}_k$. We first make the following settings:
    \begin{flalign}\nonumber
    	\begin{aligned}
    		\mathcal{A}^k&=\left\{x\in X|g_i(x,y_i)\le -\epsilon^k_i, \ \forall y_i\in Y_i, \ i\in\mathcal{V}\right\}, \\
    		\mathcal{B}^k&=\left\{x\in X|g_i(x,y_i)\le 0, \ \forall y_i\in Y_i^k, \ i\in\mathcal{V}\right\},\\
    	\end{aligned}
    \end{flalign}
    where $\epsilon^k_i$ and $Y_i^k$ are the parameters corresponding to the (${\rm ADRCP^k}$) at the $k$-th iteration of Algorithm \ref{alg:DRCP}. 
    Taking $F(x)$ as the objective function and each of the four sets $\mathcal{F}$, $\mathcal{F}^k$, $\mathcal{A}^k$ and $\mathcal{B}^k$ as the constraint, we obtain the corresponding optimal solutions $x^*$, $x^{k,*}$, $x^{k,a}$, $x^{k,b}$ by solving them in turn. Since the relation $Y^{k}_i\subset Y^{k+1}_i$ holds for any $k\ge 0$ and $i\in\mathcal{V}$, $X$ and $Y_i$ are compact, and $g_i$ are continuous on $X\times Y_i$ for any $i\in\mathcal{V}$, by Lemma 6.1 in \cite{Shapiro2009}, it holds that
    \begin{equation}\label{200000001}
    	\mathop{\lim}\limits_{k\rightarrow \infty}  F(x^{k,b})= F(x^*).
    \end{equation}	
    In addition, since $\epsilon^k_i\rightarrow 0$ as $k\rightarrow \infty$ for all agents, it follows that $\lim_{k\rightarrow\infty}\mathcal{A}^k=\mathcal{F}$.
    Based on this result, we obtain that
    \begin{flalign}\label{20000002}
    	\mathop{\lim}\limits_{k\rightarrow \infty}  F(x^{k,a})= F(x^*).
    \end{flalign}
    	According to the definitions of the sets $\mathcal{F}^k$, $\mathcal{A}^k$ and $\mathcal{B}^k$, we can obtain that  $\mathcal{A}^k\subset\mathcal{F}^{k}\subset\mathcal{B}^{k}$. Therefore, the following relation holds for any $k\ge 0$,
    	\begin{flalign}\label{100000000}
    F(x^{k,b})\le F(x^{k,*})\le F(x^{k,a})
    \end{flalign}
    By the squeeze theorem (see \cite{Petra2018}, P124) as well as the relations (\ref{200000001}), (\ref{20000002}) and (\ref{100000000}), it holds that
    \begin{equation}\label{100000001}
    \mathop{\lim}\limits_{k\rightarrow \infty}  F(x^{k,*})= F(x^*).
    \end{equation}
    By (\ref{100000001}) and the strict convexity of the function $F(x)$ (cf. Assumption \ref{Assumption 5}), it can be obtained that $x^{k,*}\rightarrow x^*$.
    \par 
    Combining the convergence result of the sequence $\{x^{k,*}\}_k$ with (\ref{10000000005}), it can be obtained that for a given $\epsilon_4>0$,
    $$\exists \overline k_i, \overline k_j: \Vert z_i^{\overline k_i+1}-z_j^{\overline k_j+1}\Vert\le \epsilon_4,\ \forall i,j\in\mathcal{V}.$$    
Hence, the consensus stopping condition (\ref{stopping_criteria_2(a)}) can be met in finite iterations of Algorithm \ref{alg:DRCP}. 
Similarly, we can verify that the last two conditions in the stopping criterion (\ref{stopping_criteria_2}) can also be satisfied within finite iterations. Therefore, we can obtain that Algorithm \ref{alg:DRCP} terminates finitely and the solutions for all agents satisfy consensus and zeroth-order stopping conditions to certain tolerances.
\end{proof}
\begin{proposition}[Local Feasibility]\label{dcsc_lf_proposition}
	Suppose Assumptions \ref{Assumption 1}, \ref{Assumption 2}, \ref{Assumption wr}--\ref{Assumption 5} are satisfied. The distributed cutting-surface consensus algorithm generates local feasible solutions for the (\ref{DRCP}), i.e., $z_i^{k+1}\in\mathcal{F}_i,\ \forall i\in\mathcal{V}$.
\end{proposition}
\begin{proof}
 In Algorithm \ref{alg:DRCP}, since $({\rm LLP}_i)$ is solved at optimality, it can be obtained that for any $i\in\mathcal{V}$, each term of the sequence $\{z_i^{k+1}\}_k$ either belongs to the set $\mathcal{F}_i$ or satisfies $\Vert z_i^{k+1}\Vert=+{\rm Inf}$. By Proposition \ref{dcsc_ftt_proposition}, we have that Algorithm \ref{alg:DRCP} terminates finitely, and $\Vert z_i^{k+1}\Vert<+{\rm Inf}$ for all $i\in\mathcal{V}$. Therefore, when Algorithm 2 terminates, the obtained point $z_i^{k+1}$ satisfies $z_i^{k+1}\in\mathcal{F}_i,\ \forall i\in\mathcal{V}$.
\end{proof}
\vspace{2mm}
\begin{remark}
	In solving the (\ref{ADRCP}), Algorithm 1 cannot provide an exact optimal solution that simultaneously satisfies the constraints of all the agents in finite time, but can only find a solution that satisfies its own constraints of each agent in finite time. Thus, the local feasibility of the solutions to the (\ref{ADRCP}) leads directly to the fact that the proposed distributed cutting-surface consensus algorithm does not return in finite time solutions that are feasible for the simultaneous satisfaction of semi-infinite constraints for all agents.
\end{remark}
\vspace{2mm}
\begin{theorem}
For each agent $i\in\mathcal{V}$, take any initial restriction parameter $\epsilon_i^0>0$, any initial finite or empty set $Y^0_i\subset Y_i$, and any reduction parameter $r> 1$. Under Assumptions \ref{Assumption 1}, \ref{Assumption 2}, \ref{Assumption wr}--\ref{Assumption 5}, the distributed cutting-surface consensus algorithm terminates in a finite number of iterations and generates locally feasible solutions for all agents satisfying consensus and zeroth-order stopping conditions to certain tolerances.
\end{theorem}
\begin{proof}
	This result follows from Propositions \ref{dcsc_ftt_proposition} and  \ref{dcsc_lf_proposition}.	
\end{proof}
\section{Numerical case studies}
In this section, a numerical case is presented to validate the effectiveness of the cutting-surface consensus algorithm (i.e., Algorithm \ref{alg:DRCP}). Firstly, we analyze and verify the finite-time termination of Algorithm \ref{alg:DRCP} and confirm the effectiveness of the distributed termination algorithm developed in this paper by comparing it with the termination algorithm in \cite{xie2017stop}. Then, we verify the local feasibility of the solutions obtained by Algorithm \ref{alg:DRCP} and compare it with the state-of-the-art methods in \cite{burger2013polyhedral,burger2012distributed}. Finally, we analyze the impacts of the restriction parameter $\epsilon^0_i$ and the reduction parameter $r$ on Algorithm \ref{alg:DRCP}. The implementation is carried out in MATLAB Version 9.5.0.944444 (R2018b, win64) and runs on an Intel(R) Core (TM) i7-7700HQ CPU @ 2.80GHz, 256GB terminal server. 
\begin{figure}[!t]
	\centerline{\includegraphics[height=2cm]{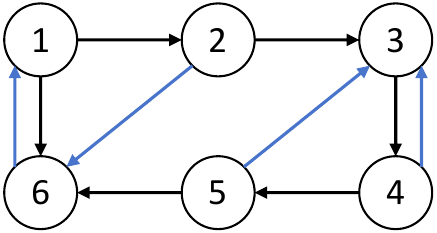}}
	\caption{Network Topology: time-varying unbalanced directed graphs, where the black and blue arrows denote the communication among agents at time $t$ and $t+1$. $\mathcal{G}(t:t+1)$ is strongly connected.}
	\label{fig2}	
\end{figure}
\par
Consider a distributed system comprised of six agents. The network topology of this system consists of time-varying unbalanced directed graphs, as shown in Figure \ref{fig2}. The optimization problem is formulated as follows. 
\begin{flalign}\label{npf_1}
	&\mathop{\min}\limits_{x\in X}  F(x)=\sum^m_{i=1}\Vert x-q_i\Vert^2
	\\&\ \text{s.t.}\ g_i(x,y_i)=(x(1)-p_i)^2+2y_ix(2)-{y_i}^2\!-\!1\!\le\! 0
	\nonumber
	\\&\qquad \qquad \qquad \qquad \qquad \qquad \ \ \forall y_i\in[-1,1]\quad i\in\mathcal{V},
	\nonumber
\end{flalign}
where $x=[x(1),x(2)]^{\mathrm{T}}$, $X=\left\{x\in\mathbb{R}^2|-2\le x(1)\le2,\right.$ $\left. -1\le  x(2)\le1\right\}$, $q_i\in \mathbb{R}^2$ and $p_i\in\mathbb{R}$ are agent $i$'s parameters for the local cost function and local constraint, respectively (see Table \ref{Table1} for specific values).
\setlength{\tabcolsep}{1mm}{
	\begin{table}[!t]
		\caption{Parameters of The Numerical Problem}
		\label{Table1}
		\centering 
		\begin{tabular}{ccccccc} 
			\toprule 
			Agent & Agent 1 & Agent 2 & Agent 3 &
			Agent 4 & Agent 5 & Agent 6  \\ 
			\midrule 
			$q_i$ & [0,6] & [0,0] & [1,1]& [-1,-1]& [1,-1]& [-1,1] \\  
			$p_i$ & -0.75 & -0.5 & -0.25 & 0.25 & 0.5 & 0.75 \\		
			\bottomrule 
		\end{tabular} 
	\end{table}
}
\subsection{Analysis of Finite-Time Termination}
\label{sec:Analysis of Finite-Time Convergence}
\begin{figure}[!t]
	\centerline{\includegraphics[width=2.8 in]{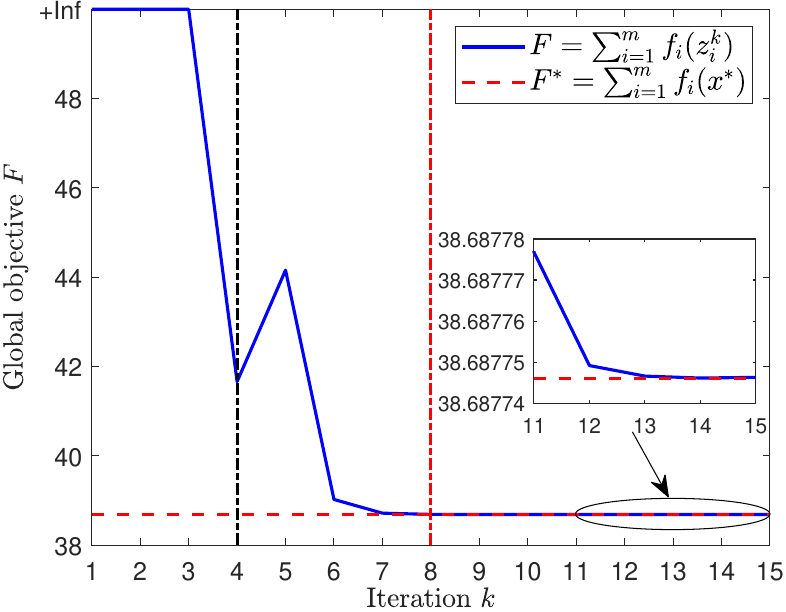}}
	\caption{Iterative process of Algorithm \ref{alg:DRCP}. The blue solid line indicates the evolution of $F=\sum^m_{i=1}f_i(z^k_i)$ with respect to iteration $k$. The red dashed line presents the global optimal value of the problem (\ref{npf_1}), where $x^*=[0,\sqrt{7}/4]^\top$, $F^*\approx 38.687746$. The horizontal coordinates of the black and red dotted-dashed lines represent the number of iterations required by the finite-time consensus termination algorithm in \cite{xie2017stop} and the termination algorithm developed in this paper, respectively.}
	\label{fig3}	
\end{figure}
For Algorithm \ref{alg:DRCP}, we initialize the restriction parameter $\epsilon^0_i\!=\!100$ and the initial set $Y^0_i\!=\!\emptyset$ for each $i\in\mathcal{V}$, along with the reduction parameter $r\!=\!10$. For Algorithm \ref{alg:3}, we set the stepsize as $\alpha(t)=\frac{1}{\sqrt{t}}$, and the termination parameters as $\epsilon_1\!=\!10^{-2}$, $\epsilon_2\!=\!10^{-6}$, $\epsilon_3\!=\!10^{-6}$. For Algorithm \ref{alg:2}, we set the termination parameters as $\epsilon_4\!=\!0.1$, $\epsilon_5\!=\!0.1$ and $\epsilon_6\!=\!0.1$. In addition, in the execution of Algorithm 1, the Euclidean projection point for each agent is obtained by solving the sub-optimization problem introduced in Remark 4, using the fmincon function in MATLAB.
\par
For this example, the distributed cutting-surface consensus algorithm terminates after eight iterations. Figure \ref{fig3} illustrates the global objective value $F\!=\!\sum^m_{i=1}f_i(z^k_i)$ at different iterations of Algorithm \ref{alg:DRCP}, from which it can be seen that the global objective $F$ converges to the optimal value $F^*$ of the problem (\ref{npf_1}) as the iterations progress. Note that in the first three iterations of Algorithm \ref{alg:DRCP}, since there exists at least an agent $i\in\mathcal{V}$ that fails to execute Step 13 in Algorithm \ref{alg:DRCP}, its candidate point $z^k_i$ satisfies $\Vert z^k_i\Vert\!=\!\Vert z^0_i\Vert\!=\!+\infty$. Therefore, we have $\sum^m_{i=1}f_i(z^k_i)=+\infty$.
\par 
In addition, we discuss the iterative process of the DPG algorithm when the local constraints $\mathcal{F}^k_i$ of all agents in the (${\rm ADRCP^k}$) are nonempty. Figure \ref{fig4} presents that the DPG algorithm reaches convergence when the (${\rm ADRCP^k}$) is feasible, i.e. $\mathcal{F}^k=\bigcap^m_{i=1}\mathcal{F}^k_i\neq\emptyset$. Figure \ref{fig5} illustrates that the DPG algorithm fails to reach consensus when the (${\rm ADRCP^k}$) is infeasible, i.e. $\mathcal{F}^k=\emptyset$. As described in Section \ref{sec:Algorithm Development and Analysis}, to ensure finite-time termination of Algorithm \ref{alg:DRCP}, we set the maximum iteration time $T=10^7$ to force the termination of the DPG algorithm for this situation.
\begin{figure}[!t]
	\centerline{\includegraphics[width=2.7 in]{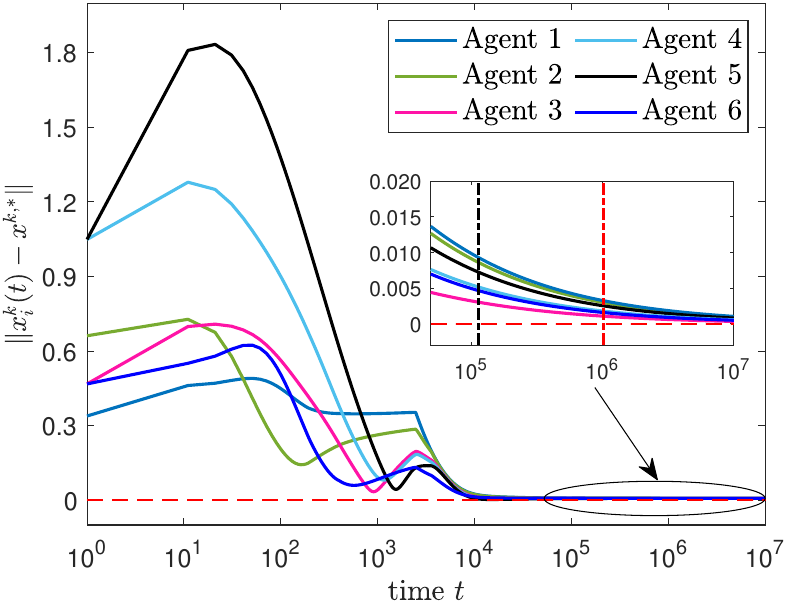}}
	\caption{Iterative process of the DPG algorithm, when the (${\rm ADRCP^k}$) is feasible ($\epsilon^k_i=0.1$ and $Y^k_i=[1]$ for all agents). The six solid lines indicates the evolution of $\Vert x^k_i(t)-x^{k,*}\Vert$ with respect to $t$, where $x^k_i(t)$ is agent $i$'s estimate generated by DPG algorithm at time $t$. The red dashed line presents the value $\Vert x^k_i(t)-x^{k,*}\Vert=0$. The horizontal coordinates of the black and red dotted-dashed lines represent the number of iterations required by the finite-time consensus termination algorithm in \cite{xie2017stop} and the proposed termination algorithm, respectively. The zoomed-in part shows our optimality gap is smaller than that achieved by \cite{xie2017stop}.}
	\label{fig4}	
\end{figure}
\begin{figure}[!t]
	\centerline{\includegraphics[width=2.7 in]{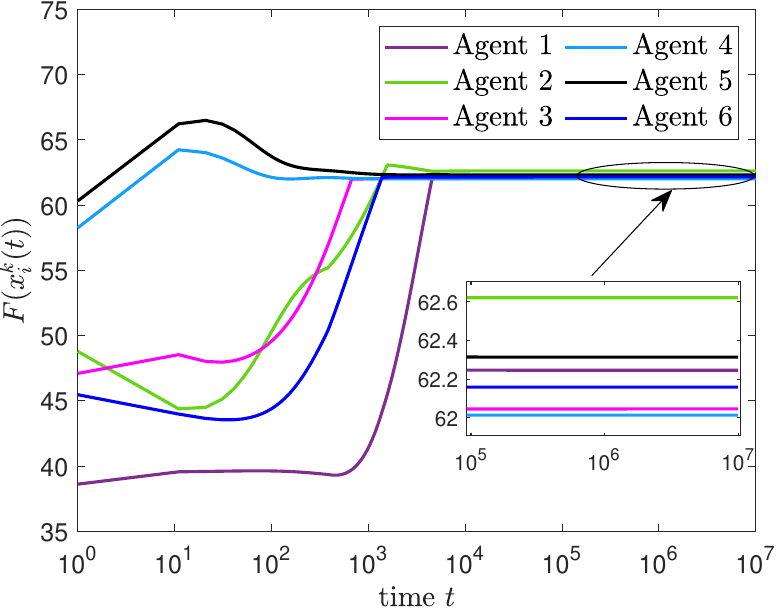}}
	\caption{Iterative process of the DPG algorithm, when the (${\rm ADRCP^k}$) is infeasible ($\epsilon^k_i=5$ and $Y^k_i=[1]$ for all agents). The six solid lines indicates the evolution of $F(x^k_i(t))$ with respect to $t$, where $x^k_i(t)$ denotes the agent $i$'s estimate generated by DPG algorithm at time $t$.}
	\label{fig5}	
\end{figure}
\par
Moreover, we compare our termination algorithm with the finite-time consensus algorithm in \cite{xie2017stop}.
In solving the (DRCP), the finite-time consensus algorithm achieves termination in the fourth iteration of Algorithm \ref{alg:DRCP}, as shown in Figure \ref{fig3}. However, the resulting solution is not the approximate optimal solution for the (DRCP). In contrast, our distributed termination algorithm not only ensures the consensus of the solutions but also guarantees their approximate optimality upon termination of Algorithm \ref{alg:DRCP}. Moreover,
when solving the (${\rm ADRCP^k}$), the optimality gap of our solution is smaller than that achieved by the finite-time consensus algorithm in \cite{xie2017stop} under the same consensus accuracy $\epsilon_1$. This stems from the introduction of the zeroth-order stopping conditions (\ref{stopping_criterion1_b}) and (\ref{stopping_criterion1_c}), enabling the DPG algorithm to generate a more approximate optimal solution for each agent, as shown in Figure \ref{fig4}.
\subsection{Analysis of Local Feasibility of The Solutions}
\label{sec:Analysis of Local Feasibility of The Solutions}
\begin{figure}[!t]
	\centerline{\includegraphics[width=2.7 in]{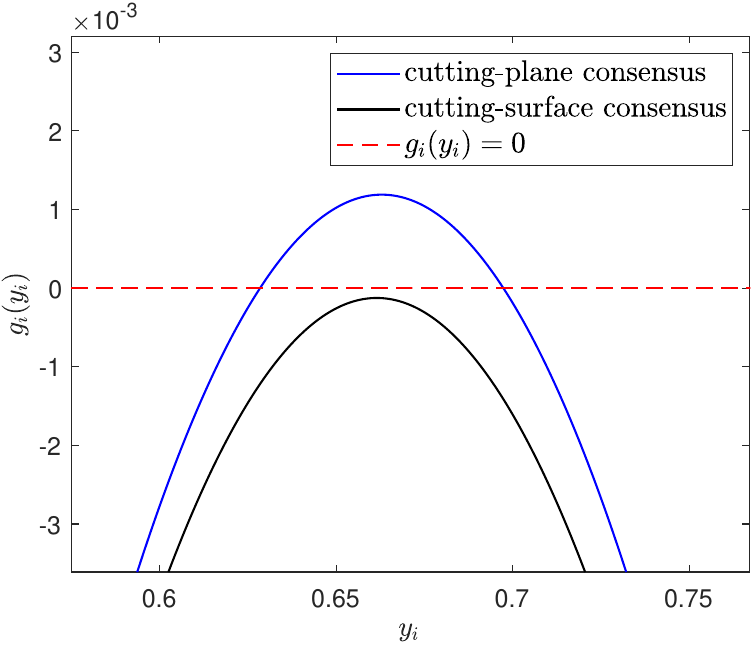}}
	\caption{Comparison of cutting-surface consensus algorithm (i.e., Algorithm \ref{alg:DRCP}) with the cutting-plane consensus algorithm in \cite{burger2012distributed,burger2013polyhedral} in terms of solution feasibility of the agent 1. The blue curve indicates the profile of $g_1(\hat z_1^k,y_1)$ regarding the uncertainty parameter $y_1$, where $\hat z_1^k$ is the resulting solution of the cutting-plane consensus algorithm. Similarly, the black curve is the profile obtained according to the cutting-surface consensus approach.}
	\label{fig6}
\end{figure}
\setlength{\tabcolsep}{1mm}{
	\begin{table}[!t]
		\caption{Local Feasibility of the solution}
		\label{Table2}
		\centering 
		\begin{tabular}{ccccccc} 
			\toprule 
			Agent & Agent 1 & Agent 2 & Agent 3 &
			Agent 4 & Agent 5 & Agent 6  \\ 
			\midrule 
			cutting-surface & \checkmark & \checkmark & \checkmark&\checkmark&\checkmark& \checkmark \\  
			cutting-plane & $\times$ &\checkmark & \checkmark& \checkmark &\checkmark & $\times$ \\		
			\bottomrule 
		\end{tabular} 
	\end{table}
} 
In this subsection, we compare Algorithm \ref{alg:DRCP} with the cutting-plane consensus algorithms in \cite{burger2012distributed,burger2013polyhedral} to present the superiority of our algorithm in terms of feasibility assurance of the resulting solutions. The cutting-plane consensus algorithms are based on successively adding linear cutting-planes into the finite set of constraints to externally approximate the ({\ref{DRCP}}), which fails to ensure the feasibility of the resulting solutions. By comparison, Algorithm \ref{alg:DRCP} is based on iteratively approximating the (\ref{DRCP}) by successively reducing the restriction parameters $\epsilon^k_i$ of the right-hand constraints and adding convex, possibly nonlinear cutting-surfaces into the existing finite sets of constraints, which enables each agent to obtain a local feasible solution within a finite number of iterations. This result is illustrated by Figure \ref{fig6} and Table \ref{Table2}.   
\subsection{Computational Performance of Adjusting $\epsilon^0_i$ and $r$}
Similar to \cite{Mitsos} and \cite{FU2015184}, the parameters $\epsilon^0_i$ and $r$ have a great influence on the computational performance of Algorithm \ref{alg:DRCP}. In this optimization problem, we set all the $\epsilon^0_i$ to equal values and $Y^0_i$ to the empty set. 
\begin{figure}[!t]
	\centerline{\includegraphics[width=2.7 in]{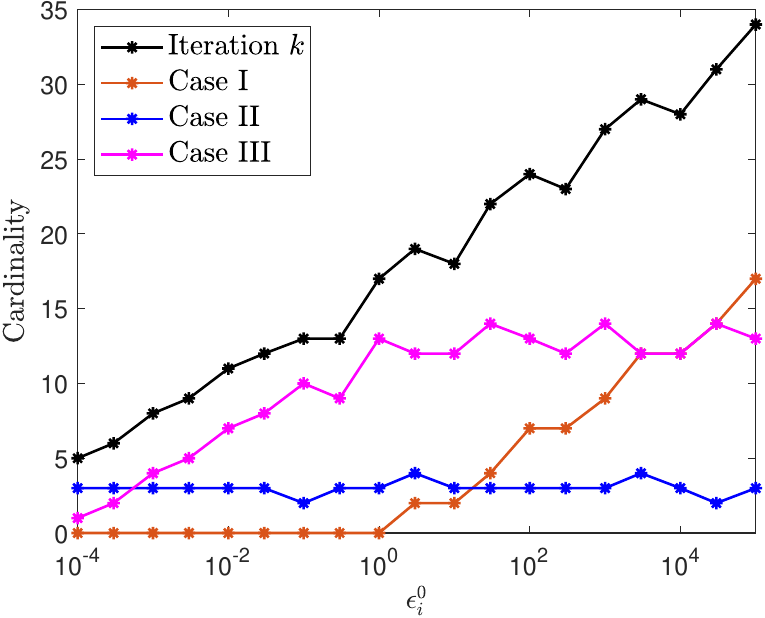}}
	\caption{Effect of tuning parameter $\epsilon^0_i$. The reduction parameter is $r\!=\!2$.}
	\label{fig7}	
\end{figure}
\begin{figure}[!t]
	\centerline{\includegraphics[width=2.7 in]{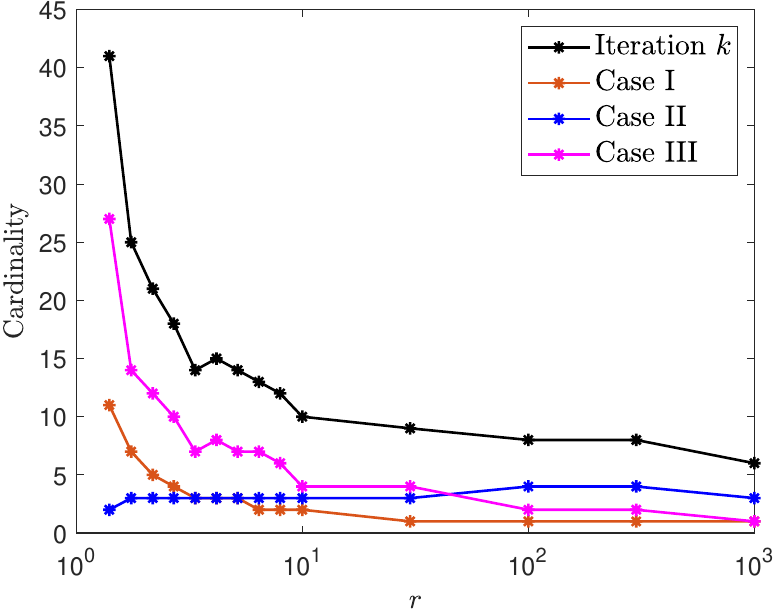}}
	\caption{Effect of tuning parameter $r$. The initial restriction parameter is $\epsilon^0_i\!=\!100$ for any $i\in\mathcal{V}$.}
	\label{fig8}	
\end{figure}
\par
 For a fixed reduction parameter $r=2$, the effect of tuning parameter $\epsilon^0_i$ is shown in Figure \ref{fig7}, where  $10^{-4}<\epsilon^0_i<10^5$. When the parameter $\epsilon^0_i$ is small ($\epsilon^0_i\le 1$), the number of Case I remains zero as the (${\rm ADRCP^k}$) is feasible throughout the iterations of Algrithm \ref{alg:DRCP}. In addition, the number of Case III sees a steady increase as $\epsilon^0_i$ rises in this stage. Conversely, when the parameter $\epsilon^0_i$ is large ($\epsilon^0_i>1$), the (${\rm ADRCP^k}$) becomes unfeasible in the initial iterations of Algorithm \ref{alg:DRCP}, leading to a growth in the number of Case I with $\epsilon^0_i$. However, the number of Case III remains basically unchanged in this stage. Moreover, the number of Case II is basically unchanged regardless of the value of $\epsilon^0_i$. To sum up, as $\epsilon^0_i$ rises, the number of iterations of Algorithm \ref{alg:DRCP} increases due to the increase in Case III (for small $\epsilon^0_i$) and Case I (for large $\epsilon^0_i$).
\par
For a fixed restriction parameter $\epsilon^0_i=100$, the effect of tuning parameter $r$ is shown in Figure \ref{fig8}, where $1<r\le 10^3$. When $r$ is small ($r<10$), increasing $r$ leads to a sharp decrease in the quantities of Case III and Case I. Conversely, for larger values of $r$ ($r>10$), the number of Case III and Case I decreases gradually and eventually stabilizes. Regardless of the value of $r$, the number of Case II remains virtually unchanged. In summary, as $r$ increases, the number of iterations of Algorithm \ref{alg:DRCP} decreases along with the reduction in the quantities of Case III and Case I.
\vspace{2mm}
\begin{remark}
	 There is a limit to the computational accuracy of MATLAB. Therefore, we cannot set $\epsilon^0_i$ too small ($\epsilon^0_i<10^{-7}$), or $r$ too large ($r>10^{3}$), as this would lead to an infinite number of Case II and iterations, thus preventing Algorithm \ref{alg:DRCP} from terminating within a finite number of iterations.
\end{remark}
\section{Conclusion}
 A distributed cutting-surface consensus approach is proposed for solving the distributed robust convex optimization problems of multi-agent systems with guaranteed local feasibility and finite-time termination, which is applicable to time-varying unbalanced directed graphs under the assumption of uniformly jointly strong connectivity. Then, it is proved that this approach enables each agent to locate a local feasible solution satisfying consensus and zeroth-order stopping conditions to certain tolerances within a finite number of iterations. Moreover, we use a numerical example to illustrate the effectiveness of the approach.
 \par
 Direct extensions may lie in the following four aspects. Firstly, in the current cutting-surface consensus approach, we use the distributed projected gradient (DPG) algorithm to solve the approximate problem (${\rm ADRCP^k}$), where the convergence rate of the DPG algorithm is sublinear. Future research could consider designing optimization algorithms with faster convergence rates (such as linear or exponential convergence) to replace the DPG algorithm, thereby improving the efficiency of the cutting-surface consensus algorithm. Secondly, the stopping criterion of this paper is not sufficient for finite-time convergence, it is sufficient for finite-time termination and enabling each agent to obtain an approximate optimal solution that satisfies consensus and zeroth-order stopping conditions to certain tolerances. Ongoing work will consider propose a novel distributed robust convex optimization algorithm that finite-time converges to a feasible solution with guaranteed sub-optimality level. In addition, it would be interesting to conduct a thorough theoretical and computational analysis of convergence rate of the distributed cutting-surface consensus algorithm. Finally, this paper mainly considers the (DRCP) with only one semi-infinite constraint for each agent. Future work will consider a more complex case that (DRCP) includes multiple semi-infinite constraints for each agent.
\section*{Acknowledgments}
We are grateful to the associate editor and the anonymous reviewers whose comments lead to substantial improvement of this article, especially Reviewers 2 and 3.
\section*{Appendix}
\label{Appendix}
In this Appendix, referring to the proof idea of \cite{mai2019}, we prove the convergence and convergence rate of the DPG algorithm proposed in Section \ref{sec:Approximation Problem}.
\subsection{Proof of Theorem 1}

\begin{lemma}[\cite{xie2018distributed} Lemma 2]\label{weighted rule lemma}
	Let $\left\{\mathcal{G}(t)\right\}$ and $\left\{A(t)\right\}$ satisfy Assumptions \ref{Assumption 2}, \ref{Assumption wr}. For any $s\ge t$, define $A(s:t)=\prod_{k=t}^{s-1}A(k)$ where $A(t:t)=I$ and $a_{ij}(s:t)$ denotes the entries of $A(s:t)$. For any $t\ge 0$, there exists a normalized vector $\pi(t)$, i.e., $1_m^\top \pi(t)=1$, such that
	\begin{itemize}
		\item[(a)]there exist $B>0$ and $0<\lambda<1$ such that $\vert a_{ij}(s:t)-\pi_j(t)\vert\le B\lambda^{s-t}$ for any $i,j\in\mathcal{V}$ and $s\ge t$.
		\item[(b)] there is a constant $\eta\ge\gamma^{(m-1)S}$ such that $\pi_i(t)\ge\eta$ for any $i\in\mathcal{V}$ and $t\ge 0$.
		\item[(c)] $(\pi(t))^\top=(\pi(t+1))^\top A(t)$.
	\end{itemize}
\end{lemma}

\begin{lemma}[\cite{Hoffmann1992} Corollary 2]\label{regularity lemma}
	Let $U$ be the convex hull of $\cup_{i\in\mathcal{V}} \Omega_i$, i.e. $U={\rm conv}(\cup_{i\in\mathcal{V}}\Omega_i)$. Suppose that Assumption \ref{Assumption ip} holds and $\Omega_i$ is a compact set, then there exists a constant $R\ge 1$ such that
	$${\rm dist}(x,\cap_{i\in\mathcal{V}} \Omega_i)\le R \max_{i\in\mathcal{V}} {\rm dist}(x,\Omega_i),\quad \forall x\in U.$$
\end{lemma}
\begin{lemma}[\cite{mai2019} Lemma 4]\label{convergence_lemma}
	Let $\left\{v(t)\right\}$, $\left\{b(t)\right\}$, $\left\{u(t)\right\}$ and $\left\{e(t)\right\}$ be nonnegative sequences, where $\sum_{t=0}^{\infty}b(t)<\infty$, $\sum_{t=0}^{\infty}e(t)<\infty$ and
	\begin{equation}\label{119}
		v(t+1)\le(1+b(t))v(t)-u(t)+e(t),\quad\forall t\ge 0.
	\end{equation}
	Then, the sequence $\left\{v(t)\right\}$ is convergent, and $\sum_{t=0}^{\infty}u(t)<\infty$.
\end{lemma}
\begin{proposition}\label{subproblem_important_result_1}
	Let Assumptions \ref{Assumption 2}--\ref{Assumption wr} be satisfied. Then for the algorithm (\ref{alg _dpg}), the following holds for any $v\in \Omega$ and $t\ge 0$:
	\begin{equation}\label{102}
		\begin{aligned}
			&\sum_{i=1}^{m}\pi_i(t+1){\Vert \theta_i(t+1)-v\Vert}^2\\
			&\le \sum_{i=1}^{m}\pi_i(t){\Vert \theta_i(t)-v\Vert}^2
			-2\alpha(t)c^\top(\overline\theta(t)-v)+C^2\alpha^2(t)\\
			&\ +2C\alpha(t)\sum_{i=1}^{m}\pi_i(t)\Vert\theta_i(t)-\overline\theta(t)\Vert-\sum_{i=1}^{m}\pi_i(t+1)\Vert\phi_i(t)\Vert^2\\
		\end{aligned}	
	\end{equation}
	where $C=\Vert c\Vert$, $\overline\theta(t)=\sum_{j=1}^{m}\pi_j(t)\theta_j(t),\ \forall t\ge 0$, and
		$\phi_i(t)=\theta_i(t+1)-\big[\sum_{j=1}^{m}a_{ij}(t)\theta_j(t)-\alpha(t)c\big].$
\end{proposition}
\begin{proof}
	Let $\tilde{\theta}_i(t)=\sum_{j=1}^{m}a_{ij}(t)\theta_j(t)$ and $v\in\Omega$. By Lemma 1(b) in \cite{nedic}, it holds that
\begin{equation}\label{104}
	\begin{aligned}
		&{\Vert\theta_i(t+1)-v\Vert}^2\!\!\!&=&{\Vert\tilde{\theta}_i(t)-v-\alpha(t)c+\phi_i(t)\Vert}^2\\
		&&\le&{\Vert\tilde{\theta}_i(t)-v-\alpha(t)c\Vert}^2\!-\!{\Vert \phi_i(t)\Vert}^2.
	\end{aligned}
\end{equation}
The first term on the right-hand side of (\ref{104}) is equal to
\begin{equation}\label{105}
	{\Vert \tilde{\theta}_i(t)-v\Vert}^2+2\alpha(t)c^\top(v-\tilde{\theta}_i(t))+\alpha^2(t) C^2.
\end{equation}
Next, we establish an upper bound for each term in (\ref{105}). Since the matrix $A(t)$ is row stochastic, it holds that $\tilde{\theta}_i(t)-v=\sum_{j=1}^{m}a_{ij}(t)(\theta_j(t)-v)$. By the convexity of $\Vert\cdot\Vert^2$, we have
\begin{equation}\label{106}
	{\Vert \tilde{\theta}_i(t)-v\Vert}^2\le \sum_{j=1}^{m}a_{ij}(t){\Vert\theta_j(t)-v\Vert}^2.
\end{equation}
Disregarding $2\alpha(t)$, the second term in (\ref{105}) meets
\begin{equation}\label{107}
	\hspace{-7mm}
	\begin{aligned}
		&&c^\top(v-\tilde{\theta}_i(t))&
		\le c^\top(v-\overline\theta(t))+\big\vert c^\top\tilde{\theta}_i(t)-c^\top\overline\theta(t)\big\vert\\
		&&&\le\! c^\top\!(v\!-\!\overline\theta(t))\!+\!C\!\sum_{j=1}^{m}\!a_{ij}(t)\Vert \theta_j(t)-\overline\theta(t)\Vert,
	\end{aligned}
\end{equation}
where the first inequality derives from the triangle inequality, and the last inequality follows from the Cauchy-Schwarz inequality and the convexity of $\Vert\cdot\Vert$.
Using the relations of (\ref{106}) and (\ref{107}), we can obtain the upper bound of (\ref{105}). Then, we take this upper bound into (\ref{104}) and execute a weighted sum on both sides of the (\ref{104}), yielding:
\begin{equation}\label{108}
	\hspace{-5mm}
	\begin{aligned}
		&&&\sum_{i=1}^{m}\pi_i(t+1){\Vert\theta_i(t+1)-v\Vert}^2\\
		&&&\le\sum_{i=1}^{m}\pi_i(t+1)\sum_{j=1}^{m}a_{ij}(t){\Vert\theta_j(t)-v\Vert}^2\\
		&&& 
		-\sum_{i=1}^{m}\pi_i(t+1)2\alpha(t)c^\top(\overline\theta(t)-v)\\
		&&& +\sum_{i=1}^{m}\pi_i(t+1)2\alpha(t)C\sum_{j=1}^{m}a_{ij}(t)\Vert \theta_j(t)-\overline\theta(t)\Vert\\
		&&& +\sum_{i=1}^{m}\pi_i(t+1)\alpha^2(t)C^2-\sum_{i=1}^{m}\pi_i(t+1){\Vert \phi_i(t)\Vert}^2.
	\end{aligned}
\end{equation}
By Lemma \ref{weighted rule lemma}(c), it holds that $\sum_{i=1}^{m}\pi_i(t\!+\!1) \sum_{j=1}^{m}a_{ij}(t)$ ${\Vert\theta_j(t)-v\Vert}^2\!=\!\sum_{i=1}^{m}\pi_i(t){\Vert(\theta_i(t)-v)\Vert}^2$, $\sum_{i=1}^{m}\pi_i(t\!+\!1)\sum_{j=1}^{m}a_{ij}(t)\Vert \theta_j(t)-\overline\theta(t)\Vert\!=\!\sum_{i=1}^{m}\pi_i(t)\Vert \theta_i(t)-\overline\theta(t)\Vert$.
Taking this result and $\sum_{i=1}^{m}\pi_i(t+1)=1$ into (\ref{108}), we can obtain (\ref{102}) as desired.
	\end{proof}
\begin{proposition}\label{subproblem_important_result_2}
	Let Assumptions \ref{Assumption 2}--\ref{Assumption wr} be satisfied. Then for the algorithm (\ref{alg _dpg}), the following relations hold.
		\begin{itemize}
		\item[(a)] Let $\beta(t)=\sum_{i=1}^{m}\Vert\phi_i(t)\Vert$. For any $t\ge0$,
			$\beta^2(t)\le m\sum_{i=1}^{m}\Vert\phi_i(t)\Vert^2$.
		\item[(b)] Let $D_1=B\sum_{j=1}^{m}\Vert\theta_j(0)\Vert$ and $D_2=mBC$. For any $i\in\mathcal{V}$, 
		$	\Vert \theta_i(t)-\overline\theta(t)\Vert\!\le\! D_1\lambda^{t}\!+\!\sum_{r=0}^{t-1}\lambda^{t-r-1}[B\beta(r)\!+\!D_2\alpha(r)]$.
		\item[(c)] Let $\psi(t)=\alpha(t)\sum_{r=0}^{t-1}\lambda^{t-r-1}\beta(r)$ with $\psi(0)=0$. If the stepsize sequence $\left\{\alpha(t)\right\}$ is nonincreasing, then
		$	\psi(t+1)\le\lambda\psi(t)+\alpha(t)\beta(t)$.
	\end{itemize}
\end{proposition}
\begin{proof}
Parts (a) and (c) are straightforward. Here, we only prove  that part (b) holds. Firstly, we rewrite (\ref{alg _dpg}) as $\theta_i(t+1)=\sum_{j=1}^{m}a_{ij}(t)\theta_j(t)+\epsilon_i(t)$, where $\epsilon_i(t)\in\mathbb{R}^n$ represents an error term. Hence, it holds that $\theta_i(t)=\sum_{j=1}^{m}[A(t:0)]_{ij}\theta_j(0)+\sum_{r=0}^{t-1}\sum_{j=1}^{m}[A(t:r+1)]_{ij}\epsilon_j(r).$ Due to $\overline\theta(t)=\sum_{j=1}^{m}\pi_j(t)\theta_j(t)$ and Lemma \ref{weighted rule lemma}(c), we obtain $\overline\theta(t)=\sum_{j=1}^{m}\pi_j(0)\theta_j(0)+\sum_{r=0}^{t-1}\sum_{j=1}^{m}\pi_j(r+1)\epsilon_j(r).$ Therefore,
\begin{equation}
	\nonumber
	\begin{aligned}
	\Vert\theta_i(t)-\overline{\theta}(t)\Vert&=\bigg\Vert\sum_{j=1}^{m}\big([A(t:0)]_{ij}-\pi_j(0)\big)\theta_j(0)\\&+\sum_{r=0}^{t-1}\sum_{j=1}^{m}\big([A(t:r+1)]_{ij}-\pi_j(r+1)\big)\epsilon_j(r)\bigg\Vert\\&\le\sum_{j=1}^{m}\big\vert[A(t:0)]_{ij}-\pi_j(0)\big\vert\Vert\theta_j(0)\Vert\\&+\sum_{r=0}^{t-1}\sum_{j=1}^{m}\big\vert[A(t:r+1)]_{ij}-\pi_j(r+1)\big\vert\Vert\epsilon_j(r)\Vert.
	\end{aligned}
	\end{equation}
  By Lemma \ref{weighted rule lemma}(a), we have
\begin{equation}\label{113}
	\Vert\theta_i(t)-\overline{\theta}(t)\Vert\!\le\! B\sum_{j=1}^{m}
	\Vert\theta_j(0)\Vert\lambda^{t}\!+\!B\sum_{r=0}^{t-1}\lambda^{t-r-1}\sum_{j=1}^{m}\Vert\epsilon_j(r)\Vert.
\end{equation}
Consider the algorithm (\ref{alg _dpg}), it satisfies that $\epsilon_i(t)=\phi_i(t)-\alpha(t)c$. By utilizing the triangle inequality, it holds that
	$\Vert\epsilon_i(t)\Vert\le\Vert\phi_i(t)\Vert+\alpha(t)C,\ \forall i\in\mathcal{V}$.
Using the bound into (\ref{113}), the result in part (b) is proved.
	\end{proof}
	\begin{corollary}\label{subproblem_important_corollary}
		In Proposition \ref{subproblem_important_result_2}, if $\lim_{t\rightarrow\infty}\beta(t)=0$, then $\lim_{t\rightarrow\infty}\psi(t)=0$. In addition, if $\lim_{t\rightarrow\infty}\alpha(t)=0$, then $\lim_{t\rightarrow\infty}\sum_{i=1}^{m}\pi_i(t)\Vert\theta_i(t)-\overline\theta(t)\Vert=0$.
	\end{corollary}
	\begin{proof}
	It can be derived from Lemma 7 in \cite{nedic}.
		\end{proof}
		\par
Based on Propositions \ref{subproblem_important_result_1} and \ref{subproblem_important_result_2}. we obtain the following important result, which is crucial for proving the convergence of the distributed projected gradient algorithm. 
	\begin{proposition}\label{subproblem_important_result_3}
		Let Assumptions \ref{Assumption 2}--\ref{Assumption wr} be satisfied. If $\left\{\alpha(t)\right\}$ is nonincreasing, then for the algorithm (\ref{alg _dpg}), we have
		\begin{equation}\label{115}
			\hspace{-2mm}
			\begin{aligned}
		&\sum_{i=1}^{m}\pi_i(t+1){\Vert \theta_i(t+1)-v\Vert}^2\!+\!ab\psi(t+1)\\
		&\le \sum_{i=1}^{m}\!\pi_i(t){\Vert \theta_i(t)\!-\!v\Vert}^2\!+\!ab\psi(t)\!-\!2\alpha(t)c^\top\!(w(t)\!-\!v)\!+\!D_5\alpha^2(t)\\
		&-\!D_4\sum_{i=1}^{m}\Vert\phi_i(t)\Vert^2\!+\!D_{13}\alpha(t)\lambda^t+D_{23}\alpha(t)\sum_{r=0}^{t-1}\lambda^{t-r-1}\alpha(r),		
			\end{aligned}	
		\end{equation}
where $w(t)=P_{\Omega}[\overline\theta(t)]$, $b=\sqrt{\frac{\eta}{m}}$, $a=\frac{D_3B}{(1-\lambda)b}$, $D_3= 2C(1+\frac{R}{\eta})$, $D_4=\frac{\eta}{2}$, $D_5=C^2+\frac{a^2}{2}$, $D_{13}=D_1D_3$ and $D_{23}=D_2D_3$.
	\end{proposition}
	\begin{proof}
By adding and subtracting $c^\top w(t)$ and using Lipschitz continuity of the function $f_0=c^\top\theta$, we have	$c^\top(v-\overline\theta(t))\le c^\top(v-w(t))+C\Vert w(t)-\overline\theta(t)\Vert$. Subsequently, we find an upper bound for the term $\Vert w(t)-\overline\theta(t)\Vert$. According to Lemma \ref{regularity lemma}, it follows that
\begin{equation}\label{116}
	\hspace{-4mm}
	\begin{aligned}
	&\Vert w(t)-\overline\theta(t)\Vert\!=\!{\rm dist}(\overline\theta(t),\Omega)\!\le\! R \max_{i\in\mathcal{V}} {\rm dist}(\overline\theta(t),\Omega_i)\\
	&\ \le\! \sum_{i=1}^{m}\frac{R\pi_i(t)}{\eta}{\rm dist}(\overline\theta(t),\Omega_i)\!\le\! \sum_{i=1}^{m}\frac{R\pi_i(t)}{\eta}\Vert\theta_i(t)\!-\!\overline\theta(t)\Vert,
	\end{aligned}
\end{equation}
where the second inequality follows from $\pi_i(t)/\eta\ge 1$ (cf. Lemma \ref{weighted rule lemma}(b)), and the last inequality is satisfied as $\theta_i(t)\in\Omega_i$. Thus, $c^\top(v-\overline\theta(t))\le c^\top(v-w(t))+\frac{RC}{\eta}\sum_{i=1}^{m}\pi_i(t)\Vert \theta_i(t)-\overline\theta(t)\Vert$. Using this bound for (\ref{102}), we obtain
\begin{equation}\label{117}
	\hspace{-4mm}
\begin{aligned}
	&\sum_{i=1}^{m}\pi_i(t+1){\Vert \theta_i(t+1)-v\Vert}^2\\
	&\le \sum_{i=1}^{m}\pi_i(t){\Vert \theta_i(t)\!-\!v\Vert}^2\!-\!2\alpha(t)c^\top(w(t)\!-\!v)\!+\!C^2\alpha^2(t)\\
	&   +D_3\alpha(t)\!\sum_{i=1}^{m}\!\pi_i(t)\Vert \theta_i(t)\!-\!\overline\theta(t)\Vert\!-\!\sum_{i=1}^{m}\pi_i(t+1)\Vert\phi_i(t)\Vert^2,
\end{aligned}
\end{equation}
where $D_3= 2C(1+\frac{R}{\eta})$. Then, adding $ab\psi(t+1)$ to both sides of (\ref{117}) and combining the results in Proposition \ref{subproblem_important_result_2}(b), (c), we further have
\begin{equation}\label{118}
	\hspace{-4mm}
	\begin{aligned}
		&\sum_{i=1}^{m}\pi_i(t+1){\Vert \theta_i(t+1)-v\Vert}^2+ab\psi(t+1)\\
		&\le \sum_{i=1}^{m}\pi_i(t){\Vert \theta_i(t)\!-\!v\Vert}^2\!+\!ab\psi(t)\!+\!ab(\lambda\!-\!1)\psi(t)\!+\!ab\alpha(t)\beta(t)\\
		&  -\!2\alpha(t)c^\top(w(t)\!-\!v)\!+\!C^2\alpha^2(t)-\!\sum_{i=1}^{m}\pi_i(t+1)\Vert\phi_i(t)\Vert^2\!\\
		& +\!D_{13}\alpha(t)\lambda^t\!+\!D_{23}\alpha({t})\sum_{r=0}^{t-1}\lambda^{t-r-1}\alpha(r)\!+\!D_3B\psi(t).
	\end{aligned}
\end{equation}
where $D_{13}=D_1D_3$, $D_{23}=D_2D_3$. Since $a=\frac{D_3B}{(1-\lambda)b}$, the two terms $ab(\lambda-1)\psi(t)$ and $D_3B\psi(t)$ cancel out. Moreover, we can also obtain $2ab\alpha(t)\beta(t)\le a^2\alpha^2(t)+b^2\beta^2(t)\le a^2\alpha^2(t)+b^2m\sum_{i=1}^{m}\Vert\phi_i(t)\Vert^2$, where the last inequality follows from Proposition \ref{subproblem_important_result_2}(a). Therefore, with $b^2=\frac{\eta}{m}$, it can be verified that $ab\alpha(t)\beta(t)-\sum_{i=1}^{m}\pi_i(t+1)\Vert\phi_i(t)\Vert^2\le \frac{a^2}{2}\alpha^2(t)-\frac{\eta}{2}\sum_{i=1}^{m}\Vert\phi_i(t)\Vert^2$. By applying the above relations to (\ref{118}), we get (\ref{115}).
	\end{proof}
\par
\par
 Finally, we prove the convergence of the algorithm (\ref{alg _dpg}), which relies on the results of Proposition \ref{subproblem_important_result_3} and Lemma \ref{convergence_lemma}.
\par
{\bf \emph{Proof of Theorem 1:}}
\par
To establish the convergence of the DPG algorithm, we mainly divide the process into two steps: (i) Applying Lemma \ref{convergence_lemma} to (\ref{115}); (ii) Proving convergence.\\
{\bf Step (i):} Set $\theta^*\in\Omega^*\ (\Omega^*\subset\Omega$) as an optimal solution and define the nonnegative sequences $\left\{v(t)\right\}$, $\left\{b(t)\right\}$, $\left\{u(t)\right\}$ and $\left\{e(t)\right\}$ as follows:
\begin{equation}\nonumber
	\begin{aligned}
	&v(t)=\sum_{i=1}^{m}\pi_i(t)\Vert\theta_i(t)-\theta^*\Vert^2+ab\psi(t),\quad b(t)=0,\\
	&u(t)=2\alpha(t)c^\top(w(t)-\theta^*)+D_4\sum_{i=1}^{m}\Vert\phi_i(t)\Vert^2,\\
	&e(t)=D_{13}\alpha(t)\lambda^t+D_{23}\alpha(t)\sum_{r=0}^{t-1}\lambda^{t-r-1}\alpha(r)+D_5\alpha^2(t).
	\end{aligned}
\end{equation}
From Proposition \ref{subproblem_important_result_3}, we can obtain $v(t+1)\le(1+b(t))v(t)-u(t)+e(t),\ \forall t\ge 0$. Then, we demonstrate that $\left\{b(t)\right\}$ and $\left\{e(t)\right\}$ are summable. Since $b(t)=0$, we have $\sum_{t=0}^{\infty}b(t)=0<\infty$. Next, consider each term in $e(t)$. Firstly, $\sum_{t=0}^{\infty}D_5\alpha^2(t)<\infty$ by Assumption \ref{Assumption sr}. Secondly, using the fact $2\alpha(t)\lambda^t\le\alpha^2(t)+\lambda^{2t}$ yields $\sum_{t=0}^{\infty}D_{13}\alpha(t)\lambda^t\le\frac{D_{13}}{2}\big[\sum_{t=0}^{\infty}\alpha^2(t)+\sum_{t=0}^{\infty}\lambda^{2t}\big]<\infty$. Thirdly, by monotonicity of $\left\{\alpha(t)\right\}$, the second term in $e(t)$ satisfies: $\alpha(t)\sum_{r=0}^{t-1}\lambda^{t-r-1}\alpha(r)\le\sum_{r=0}^{t-1}\lambda^{t-r-1}\alpha^2(r)$. According to Lemma 7 in \cite{nedic}, we can obtain $\sum_{t=0}^{\infty}\sum_{r=0}^{t-1}\lambda^{t-r-1}\alpha^2(r)<\infty$. Thus, it holds that $\sum_{t=0}^{\infty}D_{23}\alpha(t)\sum_{r=0}^{t-1}\lambda^{t-r-1}\alpha(r)<\infty$. Hence, the sequence $\left\{e(t)\right\}$ is summable. By Lemma \ref{convergence_lemma}, there exists a constant $\delta>0$ such that
\begin{align}\label{120}
	&\lim_{t\rightarrow\infty}\sum_{i=1}^{m}\pi_i(t)\Vert\theta_i(t)-\theta^*\Vert^2+ab\psi(t)=\delta,\\ \label{121}
	&\sum_{t=0}^{\infty}\bigg[2\alpha(t)c^\top(w(t)-\theta^*)+D_4\sum_{i=1}^{m}\Vert\phi_i(t)\Vert^2\bigg]<\infty.
\end{align}
{\bf Step (ii):} Firstly, by (\ref{121}), we have $\lim_{t\rightarrow\infty}\sum_{i=1}^{m}\Vert\phi_i(t)\Vert^2=0$. Thus, $\lim_{t\rightarrow\infty}\beta(t)=0$, which by Corollary \ref{subproblem_important_corollary} yields $\lim_{t\rightarrow\infty}\psi(t)=0$. It then follows from (\ref{120}) that
\begin{equation}\label{122}
	\lim_{t\rightarrow\infty}\sum_{i=1}^{m}\pi_i(t)\Vert\theta_i(t)-\theta^*\Vert^2=\delta.
\end{equation}
\par
Next, since $\sum_{t=0}^{\infty}\alpha(t)=\infty$ and $c^\top(w(t)-\theta^*)\ge 0, \forall t\ge 0$, it then follows from (\ref{121}) that $\lim_{t\rightarrow\infty}w(t)=\theta^*$.
\par
It now remains to demonstrate that $\delta=0$. By using the triangle and Cauchy-Schwarz inequalities, it can be confirmed that $\Vert \theta_i(t)-\theta^*\Vert^2\le 3(\Vert\theta_i(t)-\overline{\theta}(t)\Vert^2+\Vert\overline\theta(t)-w(t)\Vert^2+\Vert w(t)-\theta^*\Vert^2)$. Next, by the inequality (\ref{116}), we have $\Vert\overline\theta(t)-w(t)\Vert^2\le\frac{R^2}{\eta^2}\sum_{i=1}^{m}\pi_i(t)\Vert\theta_i(t)-\overline\theta(t)\Vert^2$ by the convexity of $\Vert\cdot\Vert^2$. As a result, $\frac{\Vert \theta_i(t)-\theta^*\Vert^2}{3}\le\Vert\theta_i(t)-\overline{\theta}(t)\Vert^2+\frac{R^2}{\eta^2}\sum_{i=1}^{m}\pi_i(t)\Vert\theta_i(t)-\overline\theta(t)\Vert^2+\Vert w(t)-\theta^*\Vert^2$. Multiplying both sides by $\pi_i(t)$ then summing over $i\in\mathcal{V}$ yields the follwing, with $R'=1+R^2/\eta^2$:
\begin{equation}\nonumber
\sum_{i=1}^{m}\!\frac{\pi_i(t)}{3}\Vert \theta_i(t)-\theta^*\!\Vert^2\!\le\! R'\sum_{i=1}^{m}\!\pi_i(t)\Vert\theta_i(t)-\overline\theta(t)\!\Vert^2\!+\!\Vert w(t)-\theta^*\Vert^2.
\end{equation}
Taking $\lim_{t\rightarrow\infty}$ on both sides and using (\ref{122}) yields:
\begin{equation}\label{123}
	\frac{\delta}{3}\le\lim_{t\rightarrow\infty}\Vert w(t)-\theta^*\Vert^2.
\end{equation}
Here we have use the fact $\lim_{t\rightarrow\infty}\sum_{i=1}^{m}\pi_i(t)\Vert\theta_i(t)-\overline\theta(t)\Vert=0$ (cf. Corollary \ref{subproblem_important_corollary}). Since $\delta\ge 0$ and $\lim_{t\rightarrow\infty}\Vert w(t)-\theta^*\Vert^2=0$, it holds that $\delta=0$. Therefore, we can obtain the convergence of the algorithm (\ref{alg _dpg}), i.e., $ \lim_{t\rightarrow\infty}\theta_i(t)=\theta^*,\ \forall i\in\mathcal{V}$.\hfill $\blacksquare$

\subsection{Proof of Theorem 2}
\begin{proposition}\label{subproblem_important_result_4}
Suppose that the stepsize sequence $\left\{\alpha(t)\right\}$ is nonnegative and nonincreasing and Assumptions \ref{Assumption 2}--\ref{Assumption wr} hold. Define
	$\hat\theta_i(t)=\frac{\sum_{r=0}^{t}\alpha(r)\theta_i(r)}{\sum_{r=0}^{t}\alpha(r)}$, $\hat w(t)=\frac{\sum_{r=0}^{t}\alpha(r)w(r)}{\sum_{r=0}^{t}\alpha(r)}$.
For the algorithm (\ref{alg _dpg}), the following holds:
\begin{equation}\label{125}
	C_0\Vert \hat\theta_i(t)-\hat w(t)\Vert+c^\top(\hat w(t)-\theta^*)\le E(t),\quad \forall t\ge 0,
\end{equation}
where $E(t)\!=\!\frac{C_1+C_2\sum_{r=0}^{t}\alpha^2(r)}{\sum_{\tau=0}^{t}\alpha(\tau)}$, $C_0=\frac{(1-\lambda)D_4}{2m(Rm+1)B}$, $C_1\!=\!R_1\!+\!\frac{D_{13}}{2(1-\lambda)}\alpha(0)\!+\!\frac{D_4D_1\alpha(0)}{2mB}$ and $C_2\!=\!\frac{D_{23}}{2(1-\lambda)}\!+\!\frac{D_5}{2}\!+\!\frac{D_4}{2m}(\frac{D_2}{B}\!+\!\frac{1}{4})$
with $R_1=\frac{1}{2}\sum_{i=1}^{m}\pi_i(0)\Vert\theta_i(0)-\theta^*\Vert$ for any $\theta^*\in\Omega^{*}$. Moreover,
\begin{equation}\label{126}
	\vert f_0(\hat\theta_i(t))-f_0^*\vert\le E(t)(1+C/C_0).
\end{equation}
\end{proposition}
\begin{proof}
We proceed in three steps: (i) Using Proposition \ref{subproblem_important_result_3} to give an upper estimate for the sum $\sum_{r=0}^{t}\{2\alpha(r)c^\top( w(r)-\theta^*)+D_4\sum_{i=1}^{m}\Vert\phi_i(r)\Vert^2\}$ in terms of $\sum_{r=0}^{t}\alpha(r)$ and $\sum_{r=0}^{t}\alpha^2(r)$; (ii) Relating the left side of (\ref{125}) to this sum by leveraging the convexity of $f_0$ and Proposition \ref{subproblem_important_result_2}; (iii) Proving (\ref{126}) by using Lipschitz continuity of $f_0$ and (\ref{125}).\\
{\bf Step (i):} Let $\left\{v(t)\right\}$, $\left\{u(t)\right\}$, $\left\{b(t)\right\}$ and $\left\{e(t)\right\}$ be defined as in Step (i) of the proof of Theorem \ref{subproblem_convergence_theorem}. Furthermore, set $\Phi(t)=\sum_{i=1}^{m}\Vert\phi_i(t)\Vert^2.$
By using Proposition \ref{subproblem_important_result_3}, we have that $v(t+1)\le v(t)-u(t)+e(t),\ \forall t\ge 0$, which then implies that
\begin{equation}\label{127}
	v(t+1)\le v(0)+\sum_{r=0}^{t}[e(r)-u(r)].
\end{equation}
By rearranging terms and using $v(t+1)\ge 0$, we have
\begin{equation}\label{128}
\sum_{r=0}^{t}\bigg[\alpha(r)c^\top(w(r)-\theta^*)+\frac{D_4}{2}\Phi(r)\bigg]\le R_1+\frac{1}{2}\sum_{r=0}^{t}e(r),
\end{equation}
where $R_1\!=\!\frac{1}{2}\sum_{i=1}^{m}\pi_i(0)\Vert\theta_i(0)\!-\!\theta^*\Vert$. Next, we will establish an upper bound for $\sum_{r=0}^{t}e(r)$ from the following estimates:
\begin{equation}\label{129}
\hspace{-25mm}
\sum_{r=0}^{t}\alpha(r)\lambda^r\le\sum_{r=0}^{t}\alpha(0)\lambda^r\le\frac{\alpha(0)}{1-\lambda},
\end{equation}
\begin{equation}\label{130}
	\begin{aligned}
&\sum_{r=1}^{t}\alpha(r)\sum_{s=0}^{r-1}\lambda^{r-1-s}\alpha(s)\le\sum_{r=1}^{t}\sum_{s=0}^{r-1}\lambda^{r-1-s}\alpha^2(s)\\
& =\sum_{s=0}^{t-1}\alpha^2(s)\sum_{r=s+1}^{t}\lambda^{r-1-s}\le\sum_{s=0}^{t-1}\frac{\alpha^2(s)}{1-\lambda}.
	\end{aligned}
\end{equation}
Hence, $\sum_{r=0}^{t}e(r)\le\frac{D_{13}}{1-\lambda}\alpha(0)+(\frac{D_{23}}{1-\lambda}+D_5)\sum_{r=0}^{t}\alpha^2(r)$. Therefore, we have
\begin{equation}\label{131}
	\hspace{-2mm}
\sum_{r=0}^{t}\!\bigg[\alpha(r)c^\top(w(r)-\theta^*)\!+\!\frac{D_4}{2}\Phi(r)\bigg]\!\le\! M_1\!+\!M_2\sum_{r=0}^{t}\alpha^2(r)
\end{equation}
with $M_1=R_1+\frac{D_{13}}{2(1-\lambda)}\alpha(0)$ and $M_2=\frac{1}{2}\big(\frac{D_{23}}{1-\lambda}+D_5\big)$.\\
{\bf Step (ii):} Then, we derive the lower bound on the left side of (\ref{131}). By convexity of $f_0$, it holds that
\begin{equation}\label{132}
	c^\top(\hat w(t)-\theta^*)\le\sum_{r=0}^{t}\frac{\alpha(r)c^\top(w(r)-\theta^*)}{\sum_{\tau=0}^{t}\alpha(\tau)}.
\end{equation}
Next, we relate the term $\Vert\hat\theta_i(t)-\hat{w}(t)\Vert$ to $\sum_{r=0}^{t}\Phi(r)$. By the triangle inequality, we have
\begin{equation}\label{133}
\Vert\hat\theta_i(t)-\hat{w}(t)\Vert\le\sum_{r=0}^{t}\frac{\alpha(r)\Vert\theta_i(r)-w(r)\Vert}{\sum_{\tau=0}^{t}\alpha(\tau)}.	
\end{equation}
We now quantify the numerator of the right side of (\ref{133}). Firstly, recall from the result in Proposition \ref{subproblem_important_result_2}(b) that
\begin{equation}\label{134}
	\Vert \theta_i(t)-\overline\theta(t)\Vert\le D_1\lambda^{t}+\sum_{s=0}^{t-1}\lambda^{t-s-1}[B\beta(s)+D_2\alpha(s)].
\end{equation}
Secondly, using Lemma \ref{regularity lemma} and the properties of the ${\rm dist}(\cdot)$ function, we have
\begin{equation}\label{135}
	\begin{aligned}
		&\Vert w(r)\!-\!\overline\theta(r)\Vert\!=\!{\rm dist}(\overline\theta(r),\Omega)\le R \max_{i\in\mathcal{V}} {\rm dist}(\overline\theta(r),\Omega_i)\\
		&\le R\sum_{i=1}^{m}{\rm dist}(\overline\theta(r),\Omega_i)\le R\sum_{i=1}^{m}\Vert\theta_i(r)-\overline\theta(r)\Vert,
	\end{aligned}
\end{equation}
By the triangle inequality and the inequalities (\ref{134}) and (\ref{135}), it follows that
\begin{equation}\label{136}
\frac{\Vert\theta_i(r)\!-\!w(r)\Vert}{(Rm\!+\!1)B}\!\le\!\frac{D_1}{B}\lambda^r\!+\!\sum_{s=0}^{r-1}\lambda^{r-1-s}\bigg(\frac{D_2}{B}\alpha(s)\!+\!\beta(s)\bigg),
\end{equation}
which yields (see the definition of $\psi(t)$ in Proposition \ref{subproblem_important_result_2}(c))
\begin{equation}\label{137}
\begin{aligned}
&\frac{1}{(Rm+1)B}\sum_{r=0}^{t}\alpha(r)\Vert\theta_i(r)-w(r)\Vert\\
&\le \frac{D_1}{B}\!\sum_{r=0}^{t}\alpha(r)\lambda^r\!+\!\frac{D_2}{B}\!\sum_{r=1}^{t}\alpha(r)\!\sum_{s=0}^{r-1}\lambda^{r-1-s}\alpha(s)\!+\!\sum_{r=0}^{t}\psi(r)\\
&\overset{(\ref{129}),(\ref{130})}{\le} \frac{D_1\alpha(0)}{(1-\lambda)B}+\frac{D_2}{(1-\lambda)B}\sum_{s=0}^{t}\alpha^2(s)+\sum_{r=0}^{t}\psi(r).
\end{aligned}
\end{equation}
The sum $\sum_{r=0}^{t}\psi(r)$ can be bounded as follows. Based on the result in Proposition \ref{subproblem_important_result_2}(c) and the fact that $\psi(0)=0$ and $\psi(t)\!\ge\! 0,\ \forall t\ge 1$, there is
\begin{equation}\label{138}
	\sum_{r=0}^{t}\psi(r)\le \lambda\sum_{r=0}^{t-1}\psi(r)+\sum_{r=0}^{t-1}\alpha(r)\beta(r).
\end{equation}
According to $\alpha\beta\le\frac{\alpha^2}{4}+\beta^2,\ \forall \alpha,\beta\in\mathbb{R}$, it holds that
\begin{equation}\label{139}
	\sum_{r=0}^{t}\psi(r)\le \frac{1}{1-\lambda}\sum_{r=0}^{t-1}\bigg(\frac{\alpha^2(r)}{4}+\beta^2(r)\bigg).
\end{equation}
Moreover, by Proposition \ref{subproblem_important_result_2}(a), we have $\beta^2(r)\le m\Phi(r)$. Thus
\begin{equation}\label{140}
	\sum_{r=0}^{t}\beta^2(r)\le m\sum_{r=0}^{t}\Phi(r).
\end{equation}
Using the inequalities (\ref{137}), (\ref{139}) and (\ref{140}), we obtain
\begin{equation}\label{141}
C_0\!\sum_{r=0}^{t}\alpha(r)\Vert\theta_i(r)\!-\!w(r)\Vert\!\le\! M_3\!+\!M_4\!\sum_{r=0}^{t}\alpha^2(r)\!+\!\frac{D_4}{2}\!\sum_{r=0}^{t}\Phi(r)
\end{equation}
with $C_0=\frac{(1-\lambda)D_4}{2m(Rm+1)B}$, $M_3=\frac{D_4D_1\alpha(0)}{2mB}$, $M_4=\frac{D_4}{2m}(\frac{D_2}{B}+\frac{1}{4})$. Combining the inequality (\ref{141}) with (\ref{131}) yields
\begin{equation}\label{142}
\begin{aligned}
&C_0\sum_{r=0}^{t}\alpha(r)\Vert\theta_i(r)-w(r)\Vert+\sum_{r=0}^{t}\alpha(r)c^\top(w(t)-\theta^*)\\
&\le (M_1+M_3)+(M_2+M_4)\sum_{r=0}^{t}\alpha^2(r),
\end{aligned}
\end{equation}
where $C_1=M_1+M_3$, $C_2=M_2+M_4$. Dividing both sides by $\sum_{r=0}^{t}\alpha(r)$ and then applying (\ref{132}) and (\ref{133}) leads to (\ref{125}).\\
{\bf{Step (iii):}} Since $\hat\theta_i(t)\in U$ for $\forall t\ge 0$, $\forall i\in\mathcal{V}$, it follows from triangle inequality and Lipschitz continuity of $f_0$ on $U$ that
\begin{equation}\label{143}
	\hspace{-2mm}
\begin{aligned}
	\vert c^\top (\hat\theta_i(t)-\theta^*)\vert&\le\vert c^\top (\hat\theta_i(t)-\hat w(t))\vert+c^\top(\hat w(t)-\theta^*)\\
	&\le C\Vert\hat\theta_i(t)-\hat w(t)\Vert+c^\top(\hat w(t)-\theta^*).
\end{aligned}
\end{equation}
Now, by (\ref{125}), both $C_0\Vert \hat\theta_i(t)-\hat w(t)\Vert$ and $c^\top(\hat w(t)-\theta^*)$ are bounded above by $E(t)$. Hence, (\ref{126}) is satisfied.
\end{proof}
\par
{\bf \emph{Proof of Theorem 2:}}
\par
This result follows from Proposition \ref{subproblem_important_result_4}.
\hfill $\blacksquare$
\bibliographystyle{unsrt}

\vspace{-3.9em}
\begin{IEEEbiography}[{\includegraphics[width=1in,height=1.25in,clip,keepaspectratio]{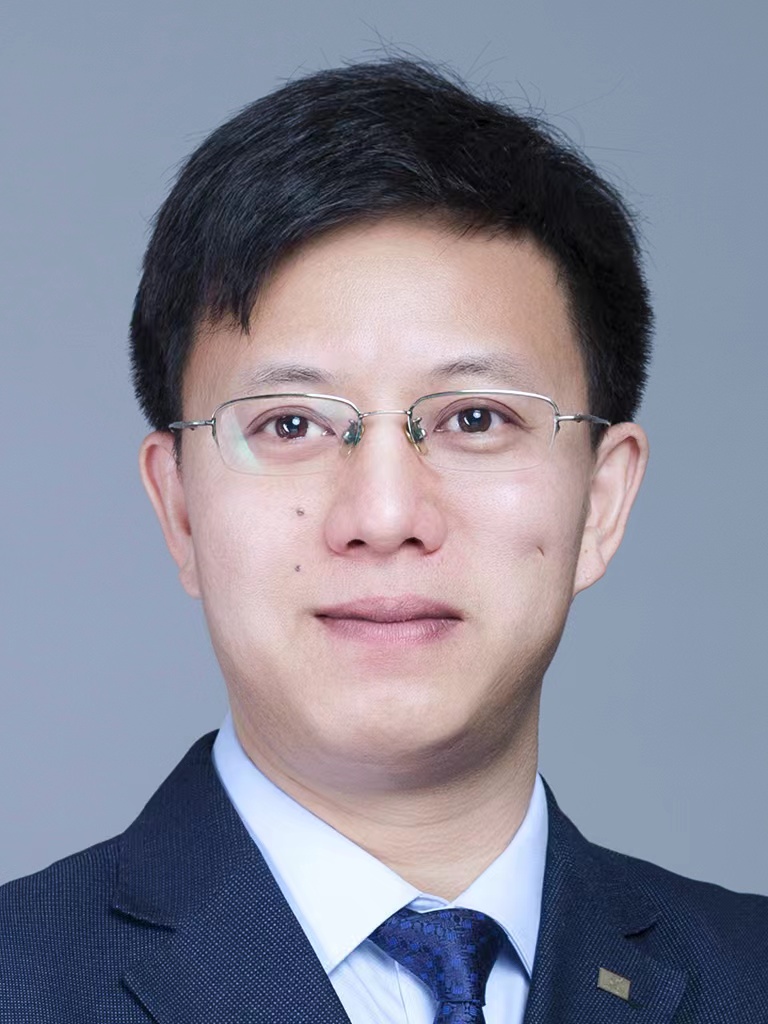}}]{Jun Fu}
	(Senior Member, IEEE) received the Ph.D. degree in mechanical engineering from Concordia University, Montreal, QC, Canada, in 2009. 
	He was a Postdoctoral Researcher with the Department of Mechanical Engineering, Massachusetts Institute of Technology (MIT), Cambridge, MA, USA, from 2010 to 2014. In 2015 he became a professor at  the State Key Laboratory of Synthetical Automation for Process Industries, Northeastern University, China, and has been working at the State Key Lab since leaving MIT. His current research is on dynamic optimization, machine learning, switched systems and their applications. He is currently an Associate Editor for the Control Engineering Practice, the IEEE Transactions on Industrial Informatics, and the IEEE Transactions on Neural Networks and Learning Systems.
\end{IEEEbiography}
\vspace{-35pt}
 \begin{IEEEbiography}[{\includegraphics[width=1in,height=1.25in,clip,keepaspectratio]{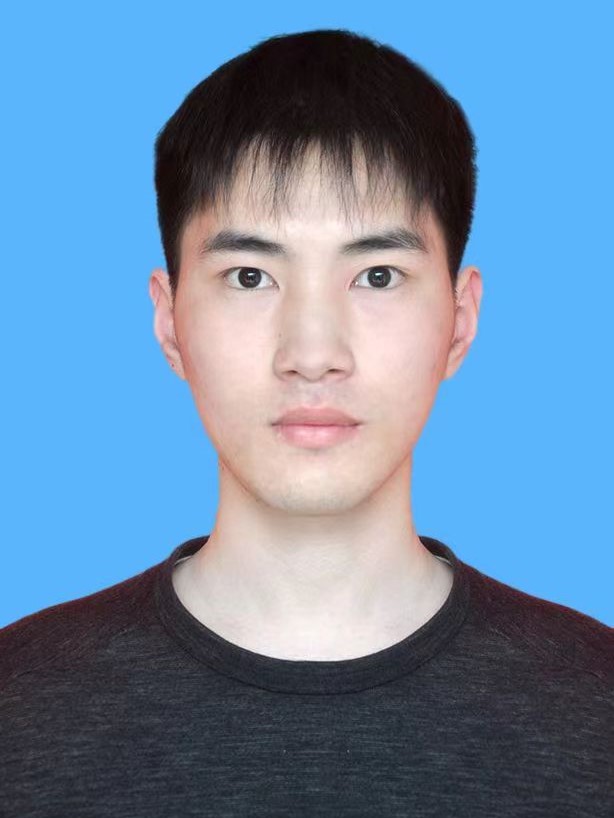}}]{Xunhao Wu}  
 	received the B.S. degree in automation from Northeastern University, Shenyang, China, in 2021, and the M.S. degree in control theory and control engineering from the State Key Laboratory of Synthetical Automation for Process Industries, Northeastern University, Shenyang, China, in 2024. He is currently pursuing the Ph.D. degree in aeronautical and aviation engineering from The Hong Kong Polytechnic University, Hong Kong, China.
 	His current research interests cover robust optimization and its application in multi-agent systems.
 \end{IEEEbiography}
\end{document}